\documentclass[10pt]{article}

\usepackage[activeacute,english]{babel}
\usepackage[utf8]{inputenc}

\usepackage{amssymb, amsmath, amsthm}
\usepackage[margin=3cm]{geometry}
\usepackage{graphicx}
\usepackage{lmodern,url}
\usepackage{color}
\usepackage{latexsym, amssymb, enumerate, amsmath}
\usepackage[]{todonotes}
\usepackage[]{empheq}
\usepackage{mathrsfs}

\newcommand{\br}{\textbf{r}}
\newcommand{\E}{\textbf{E}}
\newcommand{\U}{\textbf{U}}

\newcommand{\W}{\textbf{W}}

\newcommand{\F}{\textbf{F}}
\newcommand{\G}{\textbf{G}}

\renewcommand{\S}{\textbf{S}}
\newcommand{\nbZ}{\mathbb{Z}}
\newcommand{\nbN}{\mathbb{N}}
\newcommand{\nbR}{\mathbb{R}}

\newcommand{\om}{\overline{m}}
\newcommand{\WRP}{\W_\text{RP}}
\newcommand{\peos}{\ensuremath{p^\text{EOS}}}
\newcommand{\mround}{\ensuremath{\mathscr{M}}}

\newcommand{\IMEXLOC}[0]{$\text{IMEX}_\text{loc}$}
\newcommand{\IMEXGLOB}[0]{$\text{IMEX}_\text{glob}$}
\newcommand{\EXEXLOC}[0]{$\text{EXEX}_\text{loc}$}
\newcommand{\EXEXGLOB}[0]{$\text{EXEX}_\text{glob}$}
\newcommand{\EXEXK}[0]{$\text{EXEX}_\text{k}$}
\newcommand{\IMEXK}[0]{$\text{IMEX}_\text{k}$}

\newtheorem{lemma}{Lemma}[]
\newtheorem{proposition}{Proposition}[]

\begin{document}

\begin{center}
{\Large A large time-step and well-balanced Lagrange-Projection type scheme for the shallow-water equations} \\ ~\\

C. Chalons${}^{1}$, %
P. Kestener${}^{2}$,
S. Kokh${}^{2,3}$ 
and M. Stauffert${}^{1,2}$

\end{center}

\footnotetext[1]{Laboratoire de Math\'ematiques de Versailles, UVSQ, CNRS, Universit\'e Paris-Saclay,
78035 Versailles, France.}
\footnotetext[2]{Maison de la Simulation USR 3441, Digiteo Labs, b\^at. 565, PC 190, CEA Saclay, 91191 Gif-sur-Yvette, France.}
\footnotetext[3]{CEA/DEN/DANS/DM2S/STMF -  CEA Saclay, 91191 Gif-sur-Yvette, France.}

\begin{abstract} 
This work focuses on the numerical approximation of the Shallow Water Equations (SWE) using a Lagrange-Projection type approach. We propose to extend to this context the recent implicit-explicit schemes developed in %
\cite{ccmgsk2013}, %
\cite{ccmgsk-cicp} in the framework of compressible %
flows, with or without stiff source terms. These methods enable the use of time steps that are no longer constrained by the sound velocity thanks to an implicit treatment of the acoustic waves, and maintain accuracy in the subsonic regime thanks to an explicit treatment of the material waves. In the present setting, a particular attention will be also given to the discretization of the non-conservative terms in SWE and more specifically to the well-known well-balanced property. We prove that the proposed numerical strategy enjoys important non linear stability properties and we illustrate its behaviour past several relevant test cases. 

\end{abstract}
\makeatletter{}%
\section{Introduction}

We are interested in the design of a numerical scheme for the well-known Shallow Water Equations (SWE), given by
\begin{subequations}\label{stvenant1}
\begin{empheq}[left=\empheqlbrace]{align}
&\partial_t h + \partial_x(hu) = 0, 
 \\ 
&\partial_t (hu)+\partial_x \left( hu^2 + g \frac{h^2}{2}\right) = -gh \partial_x z,
\end{empheq}
\end{subequations}
where $z(x)$ denotes a given smooth topography and $g > 0$ is the gravity constant. 
The primitive variables are the water depth $h \geq 0$ and its velocity $u$, which both depend on the space and time variables, respectively $x \in \mathbb{R}$ and $t \in [0,\infty)$. 
At time $t=0$, we assume that the initial water depth $h(x,t=0)=h_0(x)$ and velocity $u(x,t=0)=u_0(x)$ are given.
In order to shorten the notations, we will use the following condensed form~of~\eqref{stvenant1}, namely
\begin{equation}\label{stvenant1_cond}
\partial_t \U + \partial_x \F(\U)=\S(\U,z), 
\end{equation}
where 
$\U=( h, hu )^T$, 
$\F(\U)=( hu, hu^2+gh^2/2)^T$ and
$\S(\U,z)=(0, -gh \partial_x z)^T$.
This system is supplemented with the validity of entropy inequalities which can be  written either in a non-conservative form as follows,
\begin{equation}\label{entropyIneq}
\partial_t \mathcal{U} \left( \U \right) + \partial_x \mathcal{F} \left( \U \right) \leqslant -ghu\partial_x z,
\end{equation}
with the non-conservative entropy $\mathcal{U}$ and the associated flux $\mathcal{F}$ defined by
$$
\mathcal{U}\left( \U \right) = \dfrac{hu^2}{2} + \dfrac{gh^2}{2}, \quad
\mathcal{F}\left( \U \right) = \left( \dfrac{u^2}{2} + gh \right) hu,
$$
or in conservative form 
as follows,
\begin{equation}\label{consEntropyIneq}
\partial_t \tilde{\mathcal{U}} \left( \U\,, z \right)  + \partial_x \tilde{\mathcal{F}}\left( \U\,, z \right) \leqslant 0,
\end{equation}
where the {conservative entropy} $\tilde{\mathcal{U}}$ and the associated flux $\tilde{\mathcal{F}}$ now depend on $z$ and are defined by,
$$
\tilde{\mathcal{U}} \left( \U \,, z \right) = 
  \mathcal{U}\left( \U \right) + ghz\, \quad \text{and} \quad \tilde{\mathcal{F}}\left( \U\,, z \right) = \mathcal{F}\left( \U \right) + ghuz.
$$
The proposed numerical scheme should be consistent with~\eqref{stvenant1_cond} and should satisfy a discrete form of (at least) one of these entropy inequalities. 

The steady states of \eqref{stvenant1_cond} are governed by the ordinary differential system 
$\partial_x F(\U) = \S(\U,z)$, namely 
$$
  hu = \text{constant} 
  ,\quad 
\displaystyle 
\frac{u^ 2}{2}+g(h+z)=\text{constant} 
.
$$
In this paper, we will be more specifically interested in the so-called 
"lake at rest" steady solution defined by 
\begin{equation}
  h + z = \text{constant}, \quad 
  u = 0.
  \label{genLaR}
\end{equation}
The proposed numerical scheme should be able to preserve discrete initial conditions matching \eqref{genLaR}, which corresponds to the very well-known well-balanced property (see for instance the recent book \cite{livregosse} for a review). 

A third objective of the method is to ensure the positivity 
of the water height if the initial water height is positive. 

Last but not least, we are especially interested in this work 
in subsonic or near low-Froude number flows. In this case, it turns out that 
the usual CFL time step limitation of Godunov-type numerical schemes is driven by the acoustic waves and can thus be very restrictive. We are thus interested in the design of a mixed implicit-explicit large time-step strategy following the lines of the pionneering work \cite{quote6} and the more recent ones 
\cite{ccskns2011}, \cite{ccmgsk2013}, \cite{ccmgsk2014}, \cite{ccmgsk-cicp}. By large time-step, 
we mean that the scheme should be stable under a CFL stability condition driven by the (slow) material waves, and not by the (fast) acoustic waves as it is customary in Godunov-type schemes. Numerical evidences will show a gain %
in efficiency.

There is a huge amount of works about the design of numerical schemes for the SWE, and most of the schemes intended to satisfy the first three properties above. To mention only a few of them, we refer the reader to the following well-known contributions  \cite{2}, \cite{23}, \cite{20}, \cite{25}, \cite{ps}, 
\cite{16}, \cite{17}, \cite{1}, \cite{15}, 
\cite{xing2}, \cite{12}, 
\cite{15}, \cite{8}, \cite{xing4}, \cite{xing3}, 
\cite{4}, \cite{noellelukacova}, \cite{xing1}, \cite{acu}, \cite{cb}, \cite{cornet2016}. 
We also refer to the books \cite{5} and \cite{livregosse} which provide additional  references and very nice overviews.

The design of mixed implicit-explicit (IMEX) schemes based on a Lagrange-Projection type approach which are stable under a CFL restriction driven by the slow material waves and not the acoustic waves has been given a first interest in the pionneering work \cite{quote6} and was further developed for the computation of large friction or low-Mach regimes in \cite{ccskns2011}, \cite{ccmgsk2013}, \cite{ccmgsk2014}, \cite{ccmgsk-cicp}, \cite{ccmgsk-jcp} for single or two-phase flow models. It is the purpose of this paper to adapt these IMEX strategies to the shallow-water equations while preserving the first three properties above, namely the lake-at-rest well-balanced property, the positivity of the water height, and the validity of a discrete form of the entropy inequality. 
Another new large time step method for the shallow water flows in the low
Froude number limit has been proposed in \cite{noellelukacova}. 
The strategy is also mixed implicit-explicit considering the fast acoustic waves and the slow transport waves respectively, but does not rely on the natural Lagrange-Projection like decomposition proposed here. Note also that we focus 
here on subsonic or low Froude number flows, but we do not consider the low Froude number limit which is the purpose of a current work in progress. We also refer the reader to the recent contribution \cite{hamed} which proves rigorously that the IMEX
Lagrange projection scheme is AP for one-dimensional low-Mach isentropic Euler and low-Froude shallow water equations.
\makeatletter{}%

\section{Operator splitting Lagrange-Projection approach and relaxation procedure} \label{faiiii}
In this section we adapt the so-called operator splitting Lagrange-Projection strategy presented in \cite{ccmgsk2013} %
to the Shallow Water Equations~\eqref{stvenant1}. This splitting involves a so-called Lagrange step system that accounts for the acoustic waves and topography variations for which we shall propose an approximation based on a Suliciu~\cite{suliciu1998} relaxation approach using the notion of consistency in the integral sense~\cite{16,17}, and a so-called transport step accounting for the (slow) transport phenomenon.

Before describing the numerical method, we introduce classic notations pertaining to our discretization context. Space and time are discretized using a space step $\Delta x$ and a time step $\Delta t$ into a set of cells $[x_{j-1/2},x_{j+1/2})$ and instants 
$t^{n+1} = n \Delta t$, where $x_{j+1/2}=j \Delta x$ and 
$x_j = (x_{j-1/2} + x_{j+1/2})/2$ are respectively the cell interfaces and cell centers, for $j\in\nbZ$ and $n\in\nbN$.
For a given initial condition $x\mapsto \U^0(x)$, we consider a discrete initial data $\U^{0}_{j}$ defined by
$
\U^{0}_{j} = \frac{1}{\Delta x} \int_{x_{j-1/2}}^{x_{j+1/2}}
\U^0(x)\,\text{d}x
$, for $j\in\nbZ$.
The algorithm proposed in this paper aims at computing an approximation $\U^{n}_{j}$ of $\frac{1}{\Delta x} \int_{x_{j-1/2}}^{x_{j+1/2}}
\U(x,t^n)\,\text{d}x$
where $x \to \U(x,t^n)$ is the exact solution of \eqref{stvenant1} at time $t^n$.

\subsection{Acoustic/transport operator decomposition} \label{atod}
We describe here a procedure that allows to approximate the evolution of
 the system~\eqref{stvenant1} over a time interval $[t^n,t^n+\Delta t)$.
The guideline of the method consists in decoupling the terms responsible for the acoustic, the topography variations and the transport phenomena. 
In the sequel if $ \peos: 1/h \mapsto g h^2/2$ we shall note $p = \peos(1/h)$,
and define the sound velocity $c$ of \eqref{stvenant1} by
$ c^2 = \frac{\text{d}}{\text{d} h} [(\peos)(1/h)] = {g h }$.  
By using the chain rule for the space derivatives we split up the operators of system~\eqref{stvenant1} so that it reads for smooth solutions
\begin{empheq}[left=\empheqlbrace]{align*}
&\partial_t h + u\partial_x h + h\partial_x u = 0
\\
&\partial_t (hu)  + u\partial_x (hu) + hu \partial_x u  
+ \partial_x p = -g h \partial_x z
.
\end{empheq}
Consequently, we propose to approximate the solutions of \eqref{stvenant1}
by approximating the solutions of the following two subsystems, namely
\begin{subequations}
\begin{empheq}[left=\empheqlbrace]{align}
&\partial_{t}h + h \partial_{x}u = 0,
\\
&\partial_{t}(h u) + h u\partial_{x}u + \partial_{x}p 
= - g h \partial_x z,
\end{empheq}
\label{eq: acoustic original form}
\end{subequations}
and
\begin{subequations}
\begin{empheq}[left=\empheqlbrace]{align}
&\partial_{t}h + u\partial_{x}h= 0,
\\
&\partial_{t}(h u) + u\partial_{x}(h u) = 0, 
\end{empheq}
\label{eq: transport}
\end{subequations}
one after the other. System~\eqref{eq: acoustic original form} deals with the acoustic effects and the topography variation, while system~\eqref{eq: transport} involves the material transport. In the following we shall refer to \eqref{eq: acoustic original form} as the acoustic or Lagrangian system and \eqref{eq: transport} as the transport or projection system.

The overall algorithm can be described as follows: for a given discrete state $\U^{n}_{j}=(h,hu)^{n}_{j}$, $j\in\nbZ$ that describes the system at instant $t^n$, the update to the $\U^{n+1}_j=(h,hu)^{n+1}_{j}$ is a two-step process defined by
\begin{enumerate}
\item Update $\U^{n}_{j}$ to $\U^{n+1-}_{j}$ by approximating the solution of \eqref{eq: acoustic original form},
\item Update $\U^{n+1-}_{j}$ to $\U^{n+1}_{j}$ by approximating the solution of \eqref{eq: transport}.
\end{enumerate}

\subsection{Relaxation approximation of the acoustic system}
If we note $\tau={1}/{h}$ the specific volume, the acoustic system~\eqref{eq: acoustic original form} takes the form
\begin{subequations}
\begin{empheq}[left=\empheqlbrace]{align}
\partial_{t} \tau - \tau \partial_x u &= 0,
\\
\partial_{t} u  +  \tau \partial_x p &= -g \partial_x z,
\\
\partial_t z &=0
.
\end{empheq}
\label{eq: acoustic system tau}
\end{subequations}
It is straightforward to check that the quasilinear system~\eqref{eq: acoustic system tau} is strictly hyperbolic over the space $\{(\tau,u,z)^T\in\nbR^3~|~\tau>0\}$ and the eigenstructure of the system is composed by three fields associated with the eigenvalues $\{-c,0,c\}$. The wave associated with $\pm c$ (resp. $0$) is genuinely nonlinear (resp. a stationary contact discontinuity). Let us underline that the material velocity $u$ is not involved in the characteristic velocities of \eqref{eq: acoustic system tau} but only the sound velocity $c$.

For $t\in[t^{n},t^{n}+\Delta t)$, we propose to 
approximate 
$\tau(x,t)\partial_x\cdot$ by $\tau(x,t^{n})\partial_x \cdot$ 
and
$\partial_x z  = \tau(x,t) \partial_x z / \tau(x,t)$
by
$\tau(x,t^{n}) \partial_x z / \tau(x,t)$
in \eqref{eq: acoustic system tau}. 
If one introduces the mass variable $m$ defined by $\text{d}m (x) = \tau(x,t^{n})^{-1} \text{d} x$, up to a slight abuse of notations,
system~\eqref{eq: acoustic system tau} can be recast into
\begin{subequations}
\begin{empheq}[left=\empheqlbrace]{align}
\partial_{t} \tau -  \partial_m u &= 0,
\\
\partial_{t} u  +   \partial_m p &= -g \partial_m z / \tau,
\\
\partial_t z &= 0. 
\end{empheq}
\label{eq: acoustic system m var}
\end{subequations}
Let us note that, when the topography is flat, system~\eqref{eq: acoustic system m var} is consistent 
with the usual form of the barotropic gas dynamics equations in Lagrangian coordinates with a mass space variable (see for instance \cite{godlewski-raviart-2}). 

We carry on with the approximation process of the acoustic system~\eqref{eq: acoustic original form} by using a Suliciu-type relaxation approximation of \eqref{eq: acoustic system m var}, see \cite{suliciu1998}. We will see in the sequel that this strategy will allow us to design a simple and not expensive time implicit treatment of \eqref{eq: acoustic original form} in order to remove the usual CFL restriction associated with the fast acoustic waves $\pm c$. 
The design principle of the so-called pressure relaxation methods is now very well-known, see for instance \cite{suliciu1998,quote18,quote22,quote16,quote15,quote17,5}
and the references therein and consists in introducing 
a larger system with linearly degenerate characteristic fields so that the underlying Riemann problem is easy to solve. To do so, we introduce a new independent variable pressure $\Pi$ that can be seen as a linearization of the pressure $p$.
While the pressure $p$ verifies
$
\partial_t p + (c/\tau)^2 \partial_m u = 0
$
when $\tau$ and $u$ are smooth solutions of \eqref{eq: acoustic system tau}, the surrogate pressure
 $\Pi$ is evolved according to its own partial differential equation. 
Within the time interval $t\in[t^{n},t^{n}+\Delta t)$, 
we propose to consider the following relaxation system
\begin{subequations}
\begin{empheq}[left=\empheqlbrace]{align}
\partial_{t}\tau - \partial_{m}u
&= 0
,\\
\partial_{t}u + \partial_{m}\Pi 
&= -g\partial_m z / \tau
,\\
\partial_{t}\Pi + a^{2}\partial_{m}u 
&= \lambda(p(\mathcal{T})-\Pi)
,\\
z_t &= 0
,
\end{empheq}
\label{eq: acoustic m relaxation source}
\end{subequations}
where $a>0$ is a constant whose choice will be specified later on, $\lambda>0$ is the relaxation parameter, and $\mathcal{T}$ obeys the well-defined (under appropriate conditions on $a$) implicit relation
$$
\Pi = p(\mathcal{T}) + a^2(\mathcal{T}-\tau).
$$
System~\eqref{eq: acoustic m relaxation source} is indeed an approximation of \eqref{eq: acoustic system m var}
in the sense that in the asymptotic regime
$\lambda \rightarrow +\infty$ we have, at least formally that $\Pi \rightarrow p$ and we recover \eqref{eq: acoustic system m var}, see \cite{quote15} for a rigourous proof for both smooth and discontinuous solutions. 
Let us also briefly recall that this relaxation model can be endowed with a relaxation entropy defined by
\begin{equation} \label{defsigmabbaro}
\Sigma =\frac{u^2}{2} - \int_{}^{\mathcal{T}} p^{EOS}(\tau) d\tau + 
\frac{\Pi^2-(p^{EOS})^2(\mathcal{T})}{2 \, a^2} \, ,
\end{equation}
which is such that $h \Sigma$ coincides with the entropy $\mathcal{U}$ at equlibrium $\mathcal{T}=\tau$, and by the chain rule and for smooth 
solutions easily satisfies
\begin{equation}
\partial_t \Sigma +\partial_m \Pi u =
-\frac{\lambda}{a^2} (p(\mathcal{T})-\Pi)^2 - g \frac{u}{\tau} \partial_m z,
\end{equation}
which is nothing but a relaxation and Lagrangian form of (\ref{entropyIneq}).
Note that the first term of the right-hand side is negative so that the proposed relaxation process is entropy consistent in the sense of \cite{cll}. 

We adopt the classic method that allows to reach the $\lambda\to\infty$: at each time step, we enforce the equilibrium relation
$\Pi_{i}^{n} = \peos(\tau_{i}^{n})$ and solve \eqref{eq: acoustic m relaxation source} with $\lambda=0$.
In order to prevent this relaxation procedure from generating instabilities, it is now well established that $a$ must be chosen sufficiently large in agreement with the Whitham subcharacteristic condition
\begin{equation}
a > \max \left( c(\tau) / \tau\right),
\label{eq: whitham condition}
\end{equation}
when $\tau$ spans the values 
of the solution of \eqref{eq: acoustic m relaxation source} 
for $t\in[t^{n},t^{n}+\Delta t)$ (see again the above references). 
For $\lambda = 0$, system~\eqref{eq: acoustic m relaxation source} 
can take the compact form
\begin{equation}
\partial_t \W
+\partial_m \G(\W)
=
\left(-\frac{g}{\tau}\partial_m z \right)
\mathbf{E}_2
,
\label{eq: acoustic m relaxation} 
\end{equation}
where 
$
\W=(\tau,u,\Pi,z)^T
$
,
$
\G(\W)=(-u,\Pi,a^2 u,0)^T
$
,
$
\E_2=(0,1,0,0)^T
$.
Let us discuss a few properties of \eqref{eq: acoustic m relaxation}. First,
 it can be easily proved that \eqref{eq: acoustic m relaxation} is strictly
 hyperbolic and involves  four linearly degenerate characteristic fields
 associated with the characteristic velocities $\{- a, 0, +a\}$ that are
 nothing but approximations of the eigenvalues 
 of \eqref{eq: acoustic system m var}. The jump relations involved with each 
 field are detailed in appendix~\ref{appendix: eigenstructure acoustic relaxed}.
 The non-conservative product that features in \eqref{eq: acoustic system m var} is well defined for smooth $z$ under consideration here.

Before going any further, let us observe that 
\eqref{eq: acoustic m relaxation}  can be recast into the following equivalent form
\begin{subequations}
\begin{empheq}[left=\empheqlbrace]{align}
\partial_{t}\tau - \partial_{m}u 
&= 0,
\\
\partial_{t} \overrightarrow{w} + a \partial_{m} \overrightarrow{w} 
&= - a g \partial_m z / \tau
\\
\partial_{t} \overleftarrow{w} - a \partial_{m} \overleftarrow{w} 
&=  + a g \partial_m z / \tau
,
\\
\partial_t z
&= 0
\end{empheq}
\label{eq: acoustic relaxed sym var}
\end{subequations}
where the new variables $\overrightarrow{w}$ and $\overleftarrow{w}$ are defined by 
$
\overrightarrow{w} = \Pi + au, \quad \overleftarrow{w} = \Pi - au.
$
These quantities are nothing but the strong Riemann invariants associated with the characteristic speeds 
$\pm a$ of the relaxation system (\ref{eq: acoustic relaxed sym var}) when the topography terms are omitted.
The closure relations 
for (\ref{eq: acoustic relaxed sym var}) are naturally given by 
$$
u = \frac{\overrightarrow{w}-\overleftarrow{w}}{2a}, \quad 
\Pi = \frac{\overrightarrow{w}+\overleftarrow{w}}{2}.
$$
This new formulation will be used in the sequel to study a time-implicit discretization of \eqref{eq: acoustic m relaxation}. 

We now need to propose a discretization strategy for \eqref{eq: acoustic m relaxation}. Unfortunately, the classic relaxation solver strategy cannot be carried on here since the solution of the Riemann problem associated with
\eqref{eq: acoustic m relaxation} cannot be defined easily. Indeed it is not possible to properly define the non-conservative term $\partial_m z /\tau$ with
a piecewise constant initial value for $z$. However we will see in the next section that it is 
possible to derive an approximate Riemann solver for \eqref{eq: acoustic m relaxation} using a discretization of the non-conservative product that is 
consistent (in a sense to be specified later) with the smooth term $(g/\tau)\partial_m z$. 

\subsection{Approximate Riemann solver for the acoustic system}
\label{section: approximate riemann solver}
Let $\Delta m_L>0$, $\Delta m_R>0$ and suppose given a smooth function $m \mapsto z(m)$. If $\om\in\nbR$, we consider a piecewise initial data defined by 
\begin{equation}
\W(m,t=0)
=
\begin{cases}
\W_L = (\tau_L, u_L, \Pi_L ,z_L)^T, & \text{if $m \leq \om$,}
\\
\W_R = (\tau_R, u_R, \Pi_R ,z_R)^T, & \text{if $m > \om$,}
\end{cases}
\end{equation}
where $\Pi_k$ and $z_k$, $k=L,R$ are defined by
$$
\Pi_L = \peos(\tau_L)
,\quad
\Pi_R = \peos(\tau_R)
,\quad
z_L = \frac{1}{\Delta m_L} \int^{0}_{-\Delta m_L} z(\overline{m} + r) \text{d}r
,\quad
z_R = \frac{1}{\Delta m_R} \int_{0}^{\Delta m_R} z(\overline{m} + r) \text{d}r
.
$$
Note that $\Pi_L$ and $\Pi_R$ are at equilibrium. 
Let us now build an approximate Riemann solver for the relaxed acoustic system~\eqref{eq: acoustic m relaxation}. We seek for a function $\WRP$ composed by four states separated by discontinuities as follows
\begin{equation}
\WRP
\left(
\frac{m-\om}{t}
;
\W_L, \W_R
\right)
=
\begin{cases}
\W_L,& \text{if $ \frac{m-\om}{t} \leq -a$, }
\\
\W_L^*,& \text{if $ -a  < \frac{m-\om}{t} \leq 0$, }
\\
\W_R^*,& \text{if $0 < \frac{m-\om}{t} \leq a$, }
\\
\W_R,& \text{if $ a < \frac{m-\om}{t} $, }
\end{cases}
\label{eq: approx riemann solver structure}
\end{equation}
where the intermediate states are such that the following three consistency  properties hold true (see \cite{16}, \cite{17}):

a) $\WRP$ is consistent in the integral sense with the Shallow-Water Equations, more specifically in our context:
if $\Delta t$ is such that
$a \Delta t \leq \min(\Delta m_L, \Delta m_R) / 2$, then
\begin{equation}
\G(\W_R) - \G(\W_L)
=
-a (\W_L^* - \W_L) + a (\W_R - \W_R^*)
-
\frac{\Delta m_L + \Delta m_R}{2}
\left\{\frac{g}{\tau} \partial_m z\right\},
\label{eq: integral consistency}
\end{equation}
where  $\left\{\frac{g}{\tau} \partial_m z\right\}$ is consistent with the influence of the source term, in the sense that
\begin{equation}
\lim_{\substack{
\Delta m_L, \Delta m_R \to 0
\\
\W_L, \W_R \to (\overline{\tau},\overline{u},\overline{\Pi}, z(\om))
}} 
\left\{\frac{g}{\tau} \partial_m z\right\}
=
\frac{g}{\overline{\tau}} (\partial_m z)(\om)
;
\label{eq: topo approx consistency}
\end{equation}

b) In the case of constant bottom, \text{i.e.} $z_L = z_R$, $\WRP$ must degenerate towards the classic solution of the Riemann problem of the acoustic relaxed system~\eqref{eq: acoustic m relaxation} for a flat bottom $\partial_m z = 0$;

c) If $\W_L$ and $\W_R$ verify the lake at rest condition
\begin{equation}
u_L = u_R = 0
,\qquad
1/\tau_L + z_L = 1/\tau_R + z_R,
\label{eq: lake at rest}
\end{equation}
then $\W_L^* = \W_L$ and $\W_R^* = \W_R$.

Defining a proper function $\WRP$ thus simply boils down to proposing intermediate states $\W_L^*$ and $\W^*_R$ that comply with a), b) and c). 
We proceed as follows: first we impose that $\W_L$ and $\W_L^*$ (resp. $\W_R^*$ and $\W_R$) verify the jump conditions
\begin{equation}
a (\W_L^* - \W_L) + \G(\W_L^*) - \G(\W_L) = 0
,\quad
-a (\W_R - \W_R^*) + \G(\W_R) - \G(\W_R^*) = 0
.
\label{eq: wrp extreme waves}
\end{equation}
This amounts to say that the discontinuity of velocity $\pm a$ of $\WRP$ behaves like the $(\pm) a$-wave of system~\eqref{eq: acoustic m relaxation} for a flat bottom.
Similarly, across the discontinuity of velocity $0$ we impose that 
\begin{equation}
u_L^* = u_R^* =: u^*
.
\label{eq: u_L^* = u_R^*}
\end{equation}
Relations~\eqref{eq: wrp extreme waves} and \eqref{eq: u_L^* = u_R^*} does not provide enough information to determine the intermediate states $\W_L^*$ and $\W_R^*$. Indeed, they provide us with only seven independent relations while we aim at defining eight quantities, namely the four components of each $\W_L^*$ and $\W_R^*$. 

We choose to add another jump relation across the stationary discontinuity of $\WRP$ that complies with condition b): we impose that
\begin{equation}
\Pi_R^* - \Pi_L^* + \mround = 0,
\label{eq: jump pi}
\end{equation}
where $\mround$ is a function to be specified such that $\mround=0$ if $z_L = z_R$. Relations~\eqref{eq: acoustic m relaxation}, \eqref{eq: u_L^* = u_R^*} and \eqref{eq: jump pi} allow to solve for $\W_L^*$ and $\W_R^*$ and we obtain
\begin{equation}
\left\{
\begin{aligned}
\tau_L^* &= \tau_L + \frac{1}{2a} \left(u_R-u_L\right) - \frac{1}{2a^2} \left(\Pi_R-\Pi_L\right) - \frac{\mround}{2a^2},  
\\
\tau_R^* &= \tau_R + \frac{1}{2a} \left(u_R-u_L\right) + \frac{1}{2a^2} \left(\Pi_R-\Pi_L\right) + \frac{\mround}{2a^2},  
\\
u^*  &= u_R^* = u_L^* = \frac{u_R+u_L}{2} - \frac{1}{2a} \left(\Pi_R-\Pi_L\right) 
- \frac{\mround}{2a},  
\\
\Pi^* &= \frac{\Pi_R+\Pi_L}{2} - \frac{a}{2} \left(u_R-u_L\right), 
\\
\Pi_L^* &= \Pi^* + \frac{\mround}{2},  
\\
\Pi_R^* &= \Pi^* - \frac{\mround}{2},
\\
z_L^* &= z_L,  
\\
z_R^* &= z_R.
\end{aligned}
\right.
\label{eq: intermediate states}
\end{equation}
We now only need to determine $\mround$ such that conditions a), b) and c) are satisfied. It is straightforward to see that the integral consistency requirement of condition a) implies by \eqref{eq: integral consistency} that
\begin{equation}
\mround =
\left\{\frac{g}{\tau} \partial_m z\right\}
\frac{\Delta m_L + \Delta m_R}{2}
.
\end{equation}
A simple mean to comply with conditions a) and b) is to choose
\begin{equation} \label{newlabelM}
\mround = \frac{g}{\tau_\Delta(\W_L,\W_R)}(z_R - z_L),
\end{equation}
where $\tau_\Delta(\W_L,\W_R)$ has to be chosen such that 
$\tau_\Delta(\W_L,\W_R) \to \overline{\tau}$ if $\tau_L,\tau_R\to \overline{\tau}$.
At last, we need to ensure condition c): if we have \eqref{eq: lake at rest}, then $\W_L^* = \W_L$ and $\W_R = \W_R^*$ imply that
\begin{equation}
\frac{1}{\tau_\Delta(\W_L,\W_R)}
=
\frac{1}{2}
\left(
\frac{1}{\tau_L}
+
\frac{1}{\tau_R}
\right)
.
\label{eq: def of tau delta}
\end{equation}
As a conclusion, we choose to adopt \eqref{eq: def of tau delta} as a definition of $\tau_\Delta$ for any $\W_L$ and $\W_R$. This yields that
\begin{equation}
\left\{
\frac{g}{\tau}\partial_m z
\right\}
(\W_L, \W_R, \Delta m_L, \Delta m_R)
=
\frac{ g}{\tau_\Delta(\W_L,\W_R)}
\frac{2 }{\Delta m_L + \Delta m_R}
(z_R - z_L)
.
\label{eq: def mround}
\end{equation}
It is then straightforward to check that the approximate Riemann solver defined by \eqref{eq: intermediate states} and \eqref{eq: def mround} verifies the three conditions a), b) and c). We sump up in the following proposition the properties of our Riemann solver.

\begin{proposition}
Consider the approximate Riemann solver $\WRP$ defined by \eqref{eq: approx riemann solver structure}, \eqref{eq: intermediate states}, \eqref{newlabelM} and \eqref{eq: def of tau delta}.

(i) $\WRP$ is consistent in the integral sense with the Shallow-Water Equations~\eqref{stvenant1}.

(ii) In the case of a constant bottom $z_L = z_R$, $\WRP$ degenerates to a classic approximate Riemann solver for the barotropic Euler equations in Lagrange coordinates.

(iii)  If $\W_L$ and $\W_R$ verify the lake at rest relation~\eqref{eq: lake at rest}, then $\W_k^* = \W_k$, $k=L,R$. 

\end{proposition}
\makeatletter{}%
\section{Numerical method}

In this section, we now give the details of the two-step process proposed in Section \ref{atod} for solving the Shallow Water Equations. Let us briefly recall 
that this two-step process is defined by
\begin{enumerate}
\item Update $\U^{n}_{j}$ to $\U^{n+1-}_{j}$ by approximating the solution of \eqref{eq: acoustic original form},
\item Update $\U^{n+1-}_{j}$ to $\U^{n+1}_{j}$ by approximating the solution of \eqref{eq: transport}.
\end{enumerate}
In the sequel we shall note 
$\Delta m_j = \Delta x_j h_j^n$, 
$\Delta m_{j+1/2} = (\Delta m_j + \Delta m_{j+1}) /2$, and if we assume as given 
the approximate solution $\{\U^n_j\}_j$ at time $t^n$, we introduce the approximate solution $\{\W^n_j\}_j$ at equilibrium in the $\W$ variable with a clear and natural definition.   
We begin with a fully explicit discretization of the Shallow Water Equations, which means that both steps of the process are solved with a time-explicit procedure, and we will go on with a mixed implicit-explicit strategy for which the solutions of \eqref{eq: acoustic original form} are solved implicitly in time and the solutions of \eqref{eq: transport} are solved explicitly. The latter strategy allows to get rid of the strong CFL restriction coming from the acoustic waves in the subsonic regime and corresponds to the very motivation of the present study.    
%
%
%
%
%
%
%
%
%
%

%
\makeatletter{}%
\subsection{Time-explicit discretization}

Let us begin with the time-explicit discretization of the acoustic system
\eqref{eq: acoustic original form}, or equivalently (\ref{eq: acoustic system tau}). \\
\ \\
\noindent\textbf{Acoustic step.} The acoustic update is achieved thanks to the proposed relaxation approximation and the corresponding approximate Riemann solver detailed in Section~\ref{section: approximate riemann solver}. More precisely, we propose to simply use a Godunov-type method based on this approximate Riemann solver. As it is customary and starting from the piecewise constant 
initial data defined by the sequence $\{\W^n_j\}_j$, it consists in averaging after a $\Delta t$-long time evolution, the juxtaposition of the approximate Riemann solutions defined locally at each interface $x_{j+1/2}$. Following the same lines   
as in \cite{ccmgsk2013} and \cite{ccmgsk-cicp}, see also \cite{16}, \cite{17}, \cite{CCGRS:2010} %
and the references therein, this update procedure can be easily expressed as follows after simple calculations,
\begin{subequations}
\begin{empheq}[left=\empheqlbrace]{align}
\tau^{n+1-}_{j} 
&=
 \tau^{n}_{j} + \frac{\Delta t}{\Delta m_j}
 (u_{j+1/2}^* - u_{j+1/2}^*)
 ,\\
u^{n+1-}_{j} 
&=
 u^{n}_{j} - \frac{\Delta t}{\Delta m_j}
 (\Pi_{j+1/2}^* - \Pi_{j+1/2}^*)
 -\Delta t \left\{\frac{g}{\tau}\partial_m z\right\}_j^n
 ,\\
\Pi^{n+1-}_{j} 
&=
 \Pi^{n}_{j} - \frac{\Delta t}{\Delta m_j}
 a^2(u_{j+1/2}^* - u_{j+1/2}^*)
 .
\end{empheq}
\label{eq: update acoustic explicit scheme W}
\end{subequations}

where $\Pi^n_j=p^{\mbox{EOS}}(\tau_j^n)$ and
\begin{subequations}
\label{eq: explicit acoustic numerical flux}
\begin{empheq}[left=\empheqlbrace]{align}
u_{j+1/2}^*
&=
\frac{1}{2}(u_{j}^{n} + u_{j+1}^{n})
- 
\frac{1}{2a}(\Pi_{j+1}^{n} - \Pi_{j}^{n})
-
\frac{\Delta m_{j+1/2}}{2a}
\left\{\frac{g}{\tau}\partial_m z\right\}_{j+1/2}^n,
\\
\Pi_{j+1/2}^*
&=
\frac{1}{2}(\Pi_{j}^{n} + \Pi_{j+1}^{n})
- 
\frac{a}{2}(u_{j+1}^{n} - u_{j}^{n}),
\\
\label{eq: explicit acoustic numerical flux 3}
\left\{
\frac{g}{\tau}\partial_m z
\right\}_{j+1/2}^n
&=
\left\{
\frac{g}{\tau}\partial_m z
\right\} (\W_{j}^{n}, \W_{j}^{n}, \Delta m_{j}, \Delta m_{j+1})
=
\frac{g}{2}
\left(
\frac{1}{\tau^{n}_{j}} + \frac{1}{\tau_{j+1}^{n}}
\right)
\frac{z_{j+1} - z_{j}}{\Delta m_{j+1/2}},
\\
\label{eq: explicit acoustic numerical flux 4}
\left\{
\frac{g}{\tau}\partial_m z
\right\}_{j}^n
&=
\frac{1}{2}
\left(
\frac{\Delta m_{j+1/2}}{\Delta m_j}
\left\{
\frac{g}{\tau}\partial_m z
\right\}_{j+1/2}^n
+
\frac{\Delta m_{j-1/2}}{\Delta m_j}
\left\{
\frac{g}{\tau}\partial_m z
\right\}_{j-1/2}^n
\right)
.
\end{empheq}
\end{subequations}
If we focus now on the conservative variable $\U=(h,hu)$, the discretization~\eqref{eq: update acoustic explicit scheme W} yields
the following formula for the update sequence $\{\U^{n+1-}_j\}_j$, namely
\begin{subequations}
\begin{empheq}[left=\empheqlbrace]{align}
L_j h^{n+1-} &= h^{n}_{j},
\label{eq: acoustic explicit h update}
\\
L_j (hu)^{n+1-} &= (hu)^{n}_{j} 
- \frac{\Delta t}{\Delta x_j} (\Pi^*_{j+1/2} - \Pi^*_{j-1/2})
-
\Delta t\,
h_{j}^{n}
\left\{
\frac{g}{\tau}\partial_m z
\right\}_{j}^n,
\\
L_j &= 1+\frac{\Delta t}{\Delta x_j}(u_{j+1/2}^* - u_{j-1/2}^*)
.
\label{eq: def of L_i explicit scheme}
\end{empheq}
\label{eq: acoustic update cons variable explicit}
\end{subequations}
Let us now continue with the discretization of the transport equations \eqref{eq: transport}. \\

\noindent\textbf{Transport step.} Denoting $\varphi \in \{h,hu\}$ and following again the same lines of \cite{ccmgsk2013} and \cite{ccmgsk-cicp}, see again also \cite{quote6}, we use a standard time-explicit upwind discretization for the transport step by setting
\begin{equation}
\varphi_{j}^{n+1}
=
\varphi_{j}^{n+1-}
-
\frac{\Delta t}{\Delta x_j}
(
 u_{j+1/2}^* \varphi_{j+1/2}^{n+1-} 
 -
 u_{j-1/2}^* \varphi_{j-1/2}^{n+1-}
)
+
\frac{\Delta t}{\Delta x_j}
\varphi_{j}^{n+1-} 
(
u_{j+1/2}^*
 -
u_{j-1/2}^*
)
,
\label{eq: transport scheme}
\end{equation}
where
\begin{equation*}
\varphi^{n+1-}_{j+1/2} = 
\begin{cases}
\varphi^{n+1-}_{j} ,\text{ if $u^*_{j+1/2} \geq 0$},
\\
\varphi^{n+1-}_{j+1},\text{ if $u^*_{j+1/2} < 0$}.
\end{cases}
\end{equation*}
Let us note that the transport update~\eqref{eq: transport scheme} equivalently reads
\begin{equation}
\varphi^{n+1}_j =  
\varphi^{n+1-}_j L_j
+ \frac{\Delta t}{\Delta x_j}\left( u^*_{j+1/2} \varphi^{n+1-}_{j+1/2} - u^*_{j-1/2} \varphi^{n+1-}_{j-1/2}\right),
 \label{eq: transport scheme L_j}
\end{equation}
and that the interface value of the velocity $u^*_{j+1/2}$ coincides with the one proposed in the first step, which is actually crucial in order for the whole scheme to be conservative. The next statement gather the main properties satisfied by our explicit in time and two-step algorithm. 
\\

\noindent\textbf{Overall Discretization.} 
After injecting \eqref{eq: def of L_i explicit scheme} into \eqref{eq: transport scheme L_j} one obtains the complete update procedure from $t^{n}$ to $t^{n+1}$. For the conservative variables it reads
\begin{equation} \label{eq: overal scheme explicit}
\left\{
\begin{array}{l}
h^{n+1}_{j} 
= 
h^{n}_{j}
+ \frac{\Delta t}{\Delta x_j }
\left(
 u^*_{j+1/2} h^{n+1-}_{j+1/2} - u^*_{j-1/2} h^{n+1-}_{j-1/2}
\right)
,
\\
(hu)^{n+1}_{j} 
\!= 
\!
(hu)^{n}_{j}
\!\!
+ 
\!
\frac{\Delta t}{\Delta x_j }
\!\!
\left[
 u^*_{j+1/2} (hu)^{n+1-}_{j+1/2}
 +\Pi^*_{j+1/2}
 - u^*_{j-1/2} (hu)^{n+1-}_{j-1/2}
 - \Pi^*_{j-1/2}
\right]
\!\!
+\Delta t h_j^n
\left\{
\!
\frac{g}{\tau}\partial_m z
\!
\right\}^{n}_{j}
.
\end{array}
\right.
\end{equation}

\noindent We sum up the properties of our explicit scheme \eqref{eq: update acoustic explicit scheme W}-\eqref{eq: transport scheme} in the following proposition.
\begin{proposition}
\label{GIRpart3thrm1}
The fully explicit scheme \eqref{eq: update acoustic explicit scheme W}-\eqref{eq: transport scheme} satisfies the following:

{\it (i)} it is a conservative scheme for the water height $h$. It is also a conservative scheme for $h u$ when the topography source term vanishes.

Under the Whitham subcharacteristic condition and the 
Courant-Friedrichs-Lewy (CFL) conditions, 
\begin{equation}
\max_j \frac{\Delta t}{h_{j}^{n}\Delta x_j} \leq \frac{1}{2a} \quad \mbox{and} \quad  
\max_j \frac{\Delta t}{\Delta x_j} 
\Big( (u_{j-\frac{1}{2}}^{*})^{+} - (u_{j+\frac{1}{2}}^{*})^{-} \Big) < 1
,
\label{eq: CFL explicit}
\end{equation}

{\it (ii)} the water height $h^n_j$ is positive for all $j$ and $n>0$ provided that $h^0_j$ is positive for all $j$,

{\it (iii)} it is well-balanced, with respect to the lake at rest condition~\eqref{genLaR},

{\it (iv)} it degenerates to the classic Lagrange-Projection scheme when the bottom is flat.
\end{proposition}

\begin{proof}\rule{0pt}{0pt}

(i) This is a straightforward consequence of \eqref{eq: overal scheme explicit}.

(ii) Thanks to \eqref{eq: def of L_i explicit scheme} and \eqref{eq: acoustic explicit h update}, the CFL condition~\eqref{eq: CFL explicit} ensures that $h^{n+1-}_j>0$ for $j\in\nbZ$. The CFL condition~\eqref{eq: CFL explicit} yields that $h^{n+1}_{j}$ is a convex combination of $(h^{n+1-}_{k})_{k=j\pm 1,j}$ and therefore $h^{n+1}_j >0$.

(iii) Consider a discrete fluid state at instant $t^n$ that matches the lake at rest condition, namely: 
$u_{j}^{n}$ = 0, $h_{j}^{n} + z_{j} = h_{j+1}^{n} + z_{j+1}$, for all $j\in\nbZ$. Thanks to the condition c) verified by the approximated Riemann solver of the acoustic step, we know that 
$u_{j}^{n+1-} = 0$, $h_{j}^{n+1-} + z_{j} = h_{j+1}^{n+1-} + z_{j+1}$, for all $j\in\nbZ$. And thus, the transport step~\eqref{eq: transport scheme} boils down to $h_{j}^{n+1}= h_{j}^{n+1-}$ and  $u_{j}^{n+1} = 0$.

(iv) This is consequence of condition b) imposed on the approximate Riemann solver for the acoustic step.
\end{proof}

\noindent {\it Remark.} 
Following the theory proposed by Gallice in \cite{16} and \cite{17} for non conservative systems with source terms, it is also possible to prove that our
time-explicit Godunov-type scheme based on the definition of a consistent approximate Riemann solver satisfies a discrete version 
of the non conservative entropy (\ref{entropyIneq}) under additional assumptions on the intermediate states and the propagation speed $a$. We refer for instance the reader to \cite{acu} and \cite{cb} for detailed calculations. Note that we are not able to prove at present that the scheme satisfies a discrete version of the conservative entropy inequality (\ref{consEntropyIneq}). 
%
%
%
%
%
%
%
%
%
%
%
%
%
%
%
%
%
%
%
%
%
%
%
%
%
%
%
%
%
%
%
%
%
%
%
%
%
%
%
%
%
%
%
%
%
%
%
%
%
%
%
%
%
%
%
%
%
%
%
%
%
%
%
%
%
%
%
%
%
%
%
%

%
\makeatletter{}%
\subsection{Implicit in time Lagrange-Projection method}
Let us now consider the ultimate algorithm of this paper, which consists in considering a time-implicit scheme for the Lagrangian step and keeping unchanged the transport step. As we will see in the next theorem, this strategy will allow us to obtain a non linearly stable algorithm under a CFL restriction based on the material velocity $u$ and not on the sound velocity $c$. In order to derive a time-implicit scheme for the Lagrangian step, we follow the following standard approach where the numerical fluxes are now evaluated at time $t^{n+1-}$, which gives here
the same update formulas as in the explicit case which are
\begin{subequations}
\label{eq: update acoustic implicit scheme W}
\begin{empheq}[left=\empheqlbrace]{align}
\tau^{n+1-}_{j} 
&=
 \tau^{n}_{j} + \frac{\Delta t}{\Delta m_j}
 (u_{j+1/2}^* - u_{j+1/2}^*)
 ,\\
u^{n+1-}_{j} 
&=
 u^{n}_{j} - \frac{\Delta t}{\Delta m_j}
 (\Pi_{j+1/2}^* - \Pi_{j+1/2}^*)
 -\Delta t \left\{\frac{g}{\tau}\partial_m z\right\}_j^n
 ,\\
\Pi^{n+1-}_{j} 
&=
 \Pi^{n}_{j} - \frac{\Delta t}{\Delta m_j}
 a^2(u_{j+1/2}^* - u_{j+1/2}^*)
 ,
\end{empheq}
\end{subequations}
but where the numerical fluxes now involve quantities at time $t^{n+1-}$ apart from the term consistent with $\left\{\frac{g}{\tau}\partial_m z\right\}$, which writes
\begin{subequations}
\label{eq: implicit acoustic numerical flux}
\begin{empheq}[left=\empheqlbrace]{align}
u_{j+1/2}^*
&=
\frac{1}{2}(u_{j}^{n+1-} + u_{j+1}^{n+1-})
- 
\frac{1}{2a}(\Pi_{j+1}^{n+1-} - \Pi_{j}^{n+1-})
-
\frac{\Delta m_{j+1/2}}{2a}
\left\{\frac{g}{\tau}\partial_m z\right\}_{j+1/2}^n
,
\\
\Pi_{j+1/2}^*
&=
\frac{1}{2}(\Pi_{j}^{n+1-} + \Pi_{j+1}^{n+1-})
- 
\frac{a}{2}(u_{j+1}^{n+1-} - u_{j}^{n+1-})
,
\end{empheq}
\end{subequations}
with $\left\{\frac{g}{\tau}\partial_m z\right\}_j^n$ and $\left\{\frac{g}{\tau}\partial_m z\right\}_{j+1/2}^n$ given by~\eqref{eq: explicit acoustic numerical flux 3} and~\eqref{eq: explicit acoustic numerical flux 4}.

Let us observe that we suggest here to keep on evaluating the topography source term at time $t^n$. This choice is motivated by the fact that this implicit system to be solved turns out to be a linear system with a significantly reduced coupling of the variables. More precisely, it is interesting to see that it is equivalent to the following one written in characteristic variables, namely
\begin{equation} \label{eq: implicit acoustic characteristic}
\left\{
\begin{aligned}
\tau_j^{n+1-} &= \tau_j^n + \frac{\Delta t }{\Delta m_j} \left( u_{j+1/2}^{*} - u_{j-1/2}^{*}\right), \\
\overrightarrow{w}_j^{n+1-} &= \overrightarrow{w}_j^n - a \frac{\Delta t }{\Delta m_j} \left( \overrightarrow{w}_j^{n+1-} - \overrightarrow{w}_{j-1}^{n+1-} \right) - a \Delta t \frac{\Delta m_{j-1/2}}{\Delta m_j} \left\{\frac{g}{\tau}\partial_m z\right\}_{j-1/2}^n,  \\
\overleftarrow{w}_j^{n+1-} &= \overleftarrow{w}_j^n + a \frac{\Delta t }{\Delta m_j} \left( \overleftarrow{w}_{j+1}^{n+1-} - \overleftarrow{w}_j^{n+1-} \right) + a \Delta t \frac{\Delta m_{j+1/2}}{\Delta m_j} \left\{\frac{g}{\tau}\partial_m z\right\}_{j+1/2}^n, \\
z^{n+1}_j &= z^n_j,
\end{aligned}
\right.
\end{equation}
where of course $u^*_{j+1/2}$ means here $u^{*,n+1-}_{j+1/2}$ (the notation has been lightened for the sake of clarity).
Notice that the coupling between the four variables is actually weak in (\ref{eq: implicit acoustic characteristic}) since we can easily first solve the linear system given by the second and the third equations, which are nothing but
\begin{equation*}
\left\{
\begin{aligned}
\left(I_N + a\Delta t \, A_+^n \right) \overrightarrow{w}^{n+1-} &= \overrightarrow{w}^n - a \Delta t \, b_+^n, \\
\left(I_N - a\Delta t \, A_-^n \right) \overleftarrow{w}^{n+1-} &= \overleftarrow{w}^n + a \Delta t \, b_-^n ,
\end{aligned}
\right.
\end{equation*}  
where we have set
\begin{align*}
A_+^n &= \begin{pmatrix}
\frac{1}{\Delta m_1} & 0 & \cdots & 0 \\
\frac{-1}{\Delta m_2} & \frac{1}{\Delta m_2} & \ddots & \vdots \\
0 & \ddots & \ddots & 0 \\
0 & 0 & \frac{-1}{\Delta m_N} & \frac{1}{\Delta m_N} 
\end{pmatrix},
&
b_+^n &= g \cdot \begin{pmatrix}
\frac{1}{\Delta m_1} \frac{h_1^n+h_0^n}{2} \left(z_1-z_0\right) \\ 
\frac{1}{\Delta m_2} \frac{h_2^n+h_1^n}{2} \left(z_2-z_1\right) \\ 
\vdots \\ 
\frac{1}{\Delta m_N} \frac{h_N^n+h_{N-1}^n}{2} \left(z_N-z_{N-1}\right)
\end{pmatrix},\\
\intertext{and}
A_-^n &= \begin{pmatrix}
\frac{-1}{\Delta m_1} & \frac{1}{\Delta m_1} & 0 & 0 \\
0 & \ddots & \ddots & 0 \\
\vdots & \ddots & \frac{-1}{\Delta m_{N-1}} & \frac{1}{\Delta m_{N-1}} \\
0 & \vdots & 0 & \frac{-1}{\Delta m_N} 
\end{pmatrix},
& 
b_-^n &= g \cdot \begin{pmatrix}
\frac{1}{\Delta m_1} \frac{h_2^n+h_1^n}{2} \left(z_2-z_1\right) \\ 
\vdots \\ 
\frac{1}{\Delta m_{N-1}} \frac{h_N^n+h_{N-1}^n}{2} \left(z_N-z_{N-1}\right) \\ 
\frac{1}{\Delta m_N} \frac{h_{N+1}^n+h_N^n}{2} \left(z_{N+1}-z_N\right)
\end{pmatrix}.
\end{align*}
Let us of course notice that a few coefficients of the matrices $A_+^n$ and $A_-^n$, and vectors $b_+^n$ and $b_-^n$ might be modified depending on the boundary conditions, but the purpose is to highlight that the characteristic variables $\overleftarrow{w}$ and $\overrightarrow{w}$ can be solved independently. Once this is done, the $\tau$ variable can be updated explicitly since $u^*_{j+1/2}$, or let us say $u^{*,n+1-}_{j+1/2}$, is explicitly known 
from the knowledge of $\overleftarrow{w}^{n+1-}$ and $\overrightarrow{w}^{n+1-}$ 
by the formulas 
$$
u^{n+1-}_{j} = \frac{1}{2a}(\overrightarrow{w}^{n+1-}_j - 
\overleftarrow{w}^{n+1-}_j), \quad
\Pi^{n+1-}_{j} = \frac{1}{2}(\overrightarrow{w}^{n+1-}_j + 
\overleftarrow{w}^{n+1-}_j).
$$
At last, notice that the matrices $\displaystyle \left(I_N + a\Delta t \, A_+^n \right)$ and $\displaystyle \left(I_N - a\Delta t \, A_-^n \right)$ are clealry triangular with positive diagonal coefficients, so that 
the system \eqref{eq: implicit acoustic characteristic} has a unique solution whatever the time step $\Delta t>0$ is.
It is quite natural at this stage to wonder whether the proposed time-implicit 
treatment of the Lagrangian step is well-balanced, which was true for 
the time-explicit version and was the key property leading to the well-balanced property of the global Explicit-Explicit Lagrange-Projection scheme in the previous section. It is the purpose of the next lemma.

\begin{lemma}
Under the assumption of the lake at rest at the initial time, i.e. :
\begin{equation*}
\forall j\in \left\{1,\dots,N\right\}, \ 
\left\{
\begin{aligned}
u_j^0 &= 0, \\
h_j^0 + z_j^0 &= \text{constant},
\end{aligned}
\right.
\end{equation*}
the implicit scheme for the Lagrangian step keeps this initial state unchanged, which means that the time-implicit Lagrangian step as well as the global Implicit-Explicit Lagrange-Projection scheme is still well-balanced. 
\end{lemma}

\begin{proof}\rule{0pt}{0pt}
Under the assumption of the lake at rest, and thanks to the initialisation of the relaxation pressure, namely $\forall j\in \left\{1,\dots,N\right\}$, $\displaystyle \Pi_j^0 = \frac{g}{2} \left(h_j^0\right)^2$, we get
\[
g \frac{h_j^0+h_{j-1}^0}{2}(z_j-z_{j-1}) = - \frac{g}{2}\left((h_j^0)^2 - (h_{j-1}^0)^2\right) = - (\Pi_j^0-\Pi_{j-1}^0).
\]
Thus one can write 
\[
b_+^0 = - A_+^0 \Pi^0,
\]
and, in this special case where $u^0 \equiv 0$,
\[
\left(I_N + a\Delta t \, A_+^n \right) \left(\Pi+au\right)^{1-} = \overrightarrow{w}^0 - a \Delta t \, b_+^0 = \left(\Pi+au\right)^0 + a\Delta t \, A_+^0 \Pi^0 = \left(I_N + a\Delta t \, A_+^n \right) \left(\Pi+au\right)^{0},
\]
which finally yields to
\[
 \overrightarrow{w}^{1-}=\left(\Pi+au\right)^{1-} = \left(\Pi+au\right)^{0}=\overrightarrow{w}^0.
\] 
Similarly one can prove that 
\[
\overrightarrow{w}^{1-} = \left(\Pi-au\right)^{1-} = \left(\Pi-au\right)^{0}=\overrightarrow{w}^0,
\]
so that 
\[
\left\{
\begin{aligned}
u_j^{1-} &=\frac{\overrightarrow{w}_j^{1-}-\overleftarrow{w}_j^{1-}}{2a}=\frac{\overrightarrow{w}_j^{0}-\overleftarrow{w}_j^{0}}{2a}=u_j^0=0, \\
\Pi_j^{1-} &=\frac{\overrightarrow{w}_j^{1-}+\overleftarrow{w}_j^{1-}}{2}=\frac{\overrightarrow{w}_j^{0}+\overleftarrow{w}_j^{0}}{2}=\Pi_j^0=\frac{g}{2}(h_j^0)^2,
\end{aligned}
\right.
\] 
and 
\[
u_{j+1/2}^{*,1-} = \frac{1}{2}(u_{j+1}^{1-}-u_j^{1-}) - \frac{1}{2a}(\Pi_{j+1}^{1-}-\Pi_j^{1-}) - \frac{g}{2a} \frac{h_{j+1}^0+h_j^0}{2}(z_{j+1}-z_j) = 0,
\]
for all $j$, and then 
\[
\tau_j^{1-} = \tau_j^0+\frac{\Delta t }{\Delta m_j} \left( u_{j+1/2}^{*,1-} - u_{j-1/2}^{*,1-}\right)=\tau_j^0,
\]
for all $j$. Finally we get that the lake is also at rest at the end of the Lagrangian step and, since the transport step is trivial because $u^{1^-} \equiv 0$, the global implicit-explicit scheme is well-balanced.
\end{proof}

\begin{proposition}
\label{GIRpart3thrm2}
Under the Whitham subcharacteristic condition and the CFL condition 
\begin{equation}
\max_j \frac{\Delta t}{\Delta x_j} 
\Big( (u_{j-\frac{1}{2}}^{*})^{+} - (u_{j+\frac{1}{2}}^{*})^{-} \Big) < 1
,
\label{eq: CFL implicit}
\end{equation}
the implicit-explicit scheme satisfies the following stability properties:

{\it (i)} it is a conservative scheme for the water height $h$. It is also a conservative scheme for $h u$ when the topography source term vanishes,

{\it (ii)} the water height $h^n_j$ is positive for all $j$ and $n>0$ provided that $h^0_j$ is positive for all $j$,

{\it (iii)} it is well-balanced,

{\it (iv)} it satisfies a discrete entropy inequality, 

{\it (iv)} and it gives the usual implicit-explicit Lagrange-Projection scheme when the bottom is flat.
\end{proposition}

\noindent {\bf Proof.} The properties are obtained in the same way as in the explicit case, except for the well-balanced property which has already been proved in the previous Lemma, and the validity of the entropy inequality which is proved  in appendix \ref{appendix:proof_proposition_implicite}.
\makeatletter{}%
\section{Numerical results}

\renewcommand{\thetable}{\Alph{table}}

The aim of this section is to illustrate the behaviour 
of our Lagrange-Projection like strategies in one space dimension. 
We will also compare the results with the simple, well-balanced, positive and entropy-satisfying scheme recently proposed in \cite{acu} (the scheme will be referred to as the HLLACU scheme) and the very well-known hydrostatic reconstruction scheme \cite{1} based on a classic HLL scheme and referred to as HRHLL in the following.

Let us first notice that two (classic) options will be considered in order to evaluate 
the artificial sound speed $a$ involved in the acoustic step. Let 
$\kappa > 1$. 
The first one is based on a local definition of the Lagrangian sound speed 
in agreement with a local evaluation of the subcharacteristic condition, namely
\begin{equation}
a_{j+1/2} = \kappa \max \left(h_j^n \sqrt{gh_j^n}, h_{j+1}^n \sqrt{gh_{j+1}^n}\right)
\label{eq: local choice of a}, \quad \forall \, j \, \in \mathbb{Z},
\end{equation}
while the second one considers an uniform estimate by setting
\begin{equation}
a_{j+1/2} =
\kappa \max_{i\in\nbZ} \left(h_i^n \sqrt{gh_i^n}\right), \quad \forall \, j \, \in \mathbb{Z}.
\label{eq: uniform choice of a}
\end{equation}
In practice, we set $\kappa = 1.01$.
For the sake of conciseness, the full-explicit scheme will be referred to as \EXEXLOC{} (resp. \EXEXGLOB{}) and the semi-implicit scheme will be referred to as \IMEXLOC{} (resp. \IMEXGLOB{}) when \eqref{eq: local choice of a} (resp. \eqref{eq: uniform choice of a}) is used.

Let us also mention for all the test cases, uniform space steps $\Delta x$ will be considered and the time steps $\Delta t$ will be chosen in agreement with the 
CFL conditions \eqref{eq: CFL explicit} and \eqref{eq: CFL implicit}. 
More precisely, we will set (unless otherwise stated) 
\begin{equation} \label{cflnum1}
\Delta t = \frac{\Delta x}{2 \max_{j \in \mathbb{Z}} (\sqrt{g {h}_j^n},\lvert u_{j+1/2}^{*} \rvert)},
\end{equation}
for the explicit schemes, and  
\begin{equation} \label{cflnum2}
\Delta t = \frac{\Delta x}{2 \max_{j \in \mathbb{Z}} (\lvert u_{j+1/2}^{*} \rvert)}, \end{equation}
for the implicit ones,
where $u_{j+1/2}^{*}$ is calculated at time $t^n$ for the sake of simplicity.

Before starting, let us finally mention that initial data matching the lake at rest condition~\eqref{genLaR} are preserved by construction by the \EXEXK{} and \IMEXK{} schemes, $k=\mbox{loc}, \mbox{glob}$. Therefore, such test cases will not be considered hereafter.

\subsection{Dam break problem}\label{section: dam break test}
We first consider the classical dam break. 
The space domain $[0,1500]$ is divided into two parts with the same length 
and such that the water height is higher on the left side,
\[
h(x,t=0)=20,\quad\text{if $x\leq 750$}
,\quad
h(x,t=0)=15,\quad\text{if $x > 750$.}
\]
The velocity is set to be zero on both side at the initial time when the dam breaks and the water starts flowing.
Importantly, the topography is not flat but given by the regularized two-step 
function 
\[
z(x) = \begin{cases}
4 e^{2-\frac{150}{x-487.5}}, & \text{ if } 487.5<x<=562.5, \\
8 - 4 e^{2-\frac{150}{637.5-x}}, & \text{ if } 562.5<x<=637.5, \\
8, & \text{ if } 637.5<x<=862.5, \\
8 - 4 e^{2-\frac{150}{x-862.5}}, & \text{ if } 862.5<x<=937.5, \\
4 e^{2-\frac{150}{1012.5 -x}}, & \text{ if } 937.5<x<=1012.5, \\
0 & \text{ otherwise.}
\end{cases}
\]
At last, the spatial domain is discretized over a $1500$-cell grid and Neumann boundary conditions are used.
Figures~\ref{fig: DamBreak} and ~\ref{fig: DamBreak} show the solutions at final times 
$T=10$ and $T=50$ with different numerical strategies. %
The following comments are in order. We first observe that the implicit schemes are the most diffusive, which was clearly expected from the implicit treatment of the acoustic step. Note also that our Lagrangian-Projection schemes are intrinsically 
made of two averaging steps, which is necessary to separate the acoustic and transport effects, but at the price of additional numerical diffusion compared to a direct Eulerian approach like the one proposed in the HLLACU scheme. We also observe that a local definition of parameter $a$ is preferable to the global one in order to reduce the numerical diffusion. For this reason, we will only consider the local evaluation in the following test cases ($k=\mbox{loc}$). As far as the time step step is concerned,
we observed for this test case that the averaged value (calculated from the time iterations needed to reach the final time $T=50$) is about five times larger for the 
\IMEXLOC{} than for the \EXEXLOC{} schemes.

\begin{figure}
\centering
\begin{tabular}{cc}
\includegraphics[width=.49\textwidth]{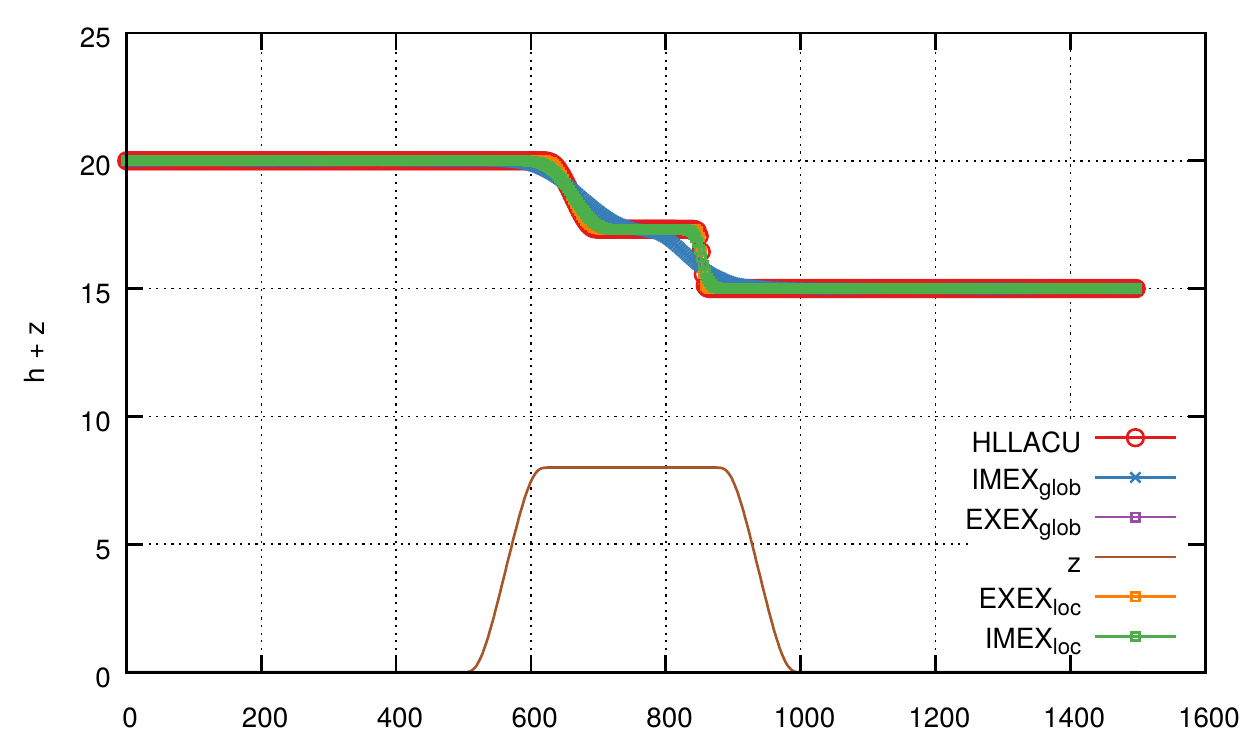}
&
\includegraphics[width=.49\textwidth]{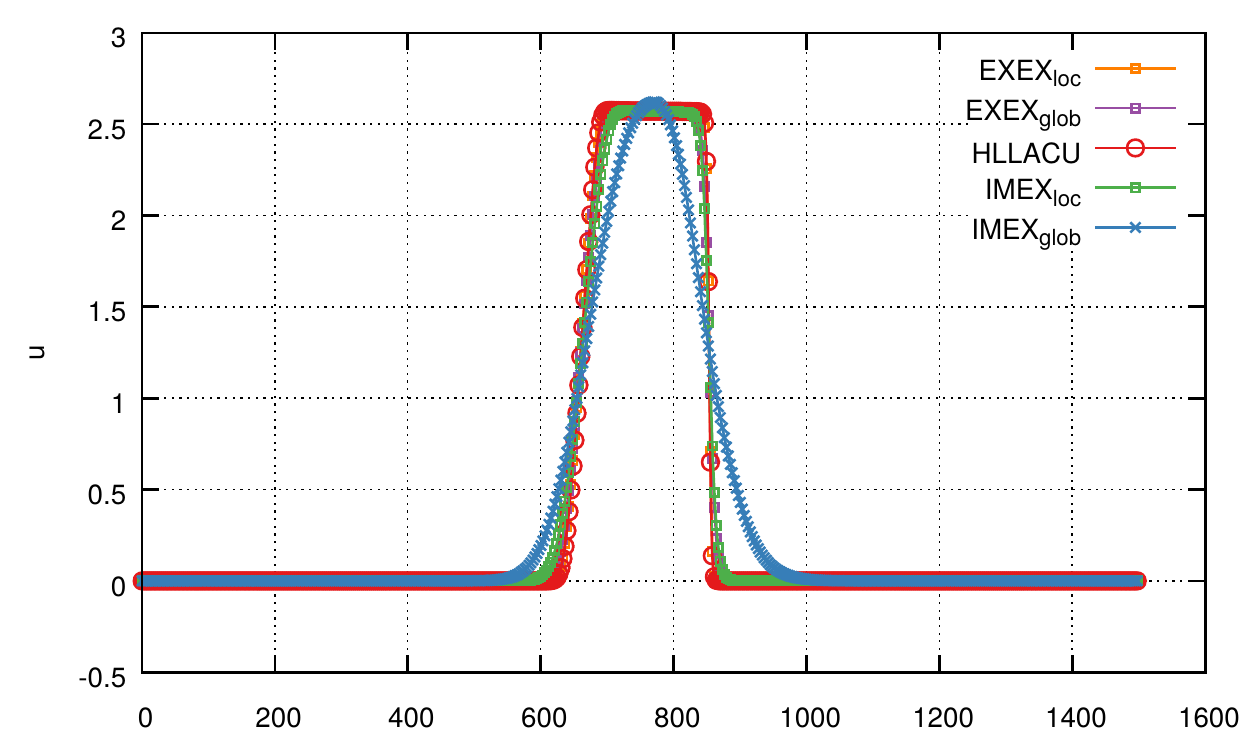}
\end{tabular}\caption{Dam Break problem. Profile of $z+h$ and $u$ at time $T=10$.}
\label{fig: DamBreak}
\end{figure}

\begin{figure}
\centering
\begin{tabular}{cc}
\includegraphics[width=.49\textwidth]{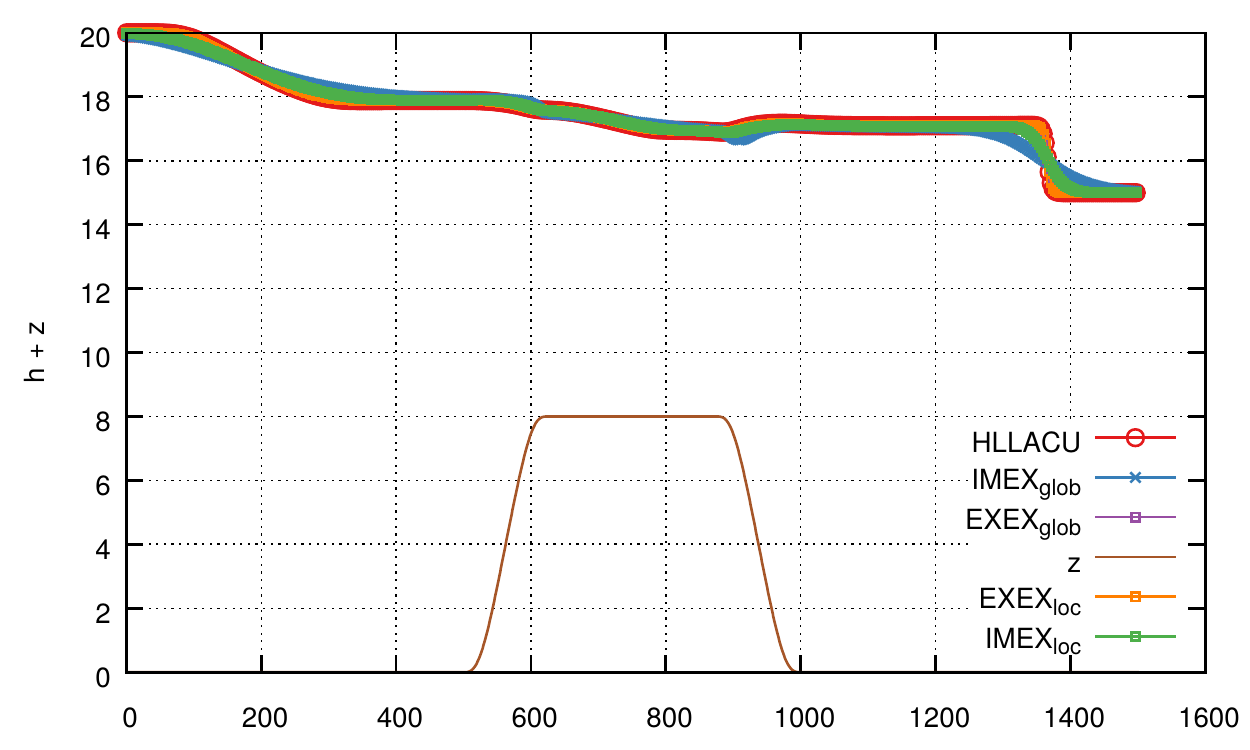}
&
\includegraphics[width=.49\textwidth]{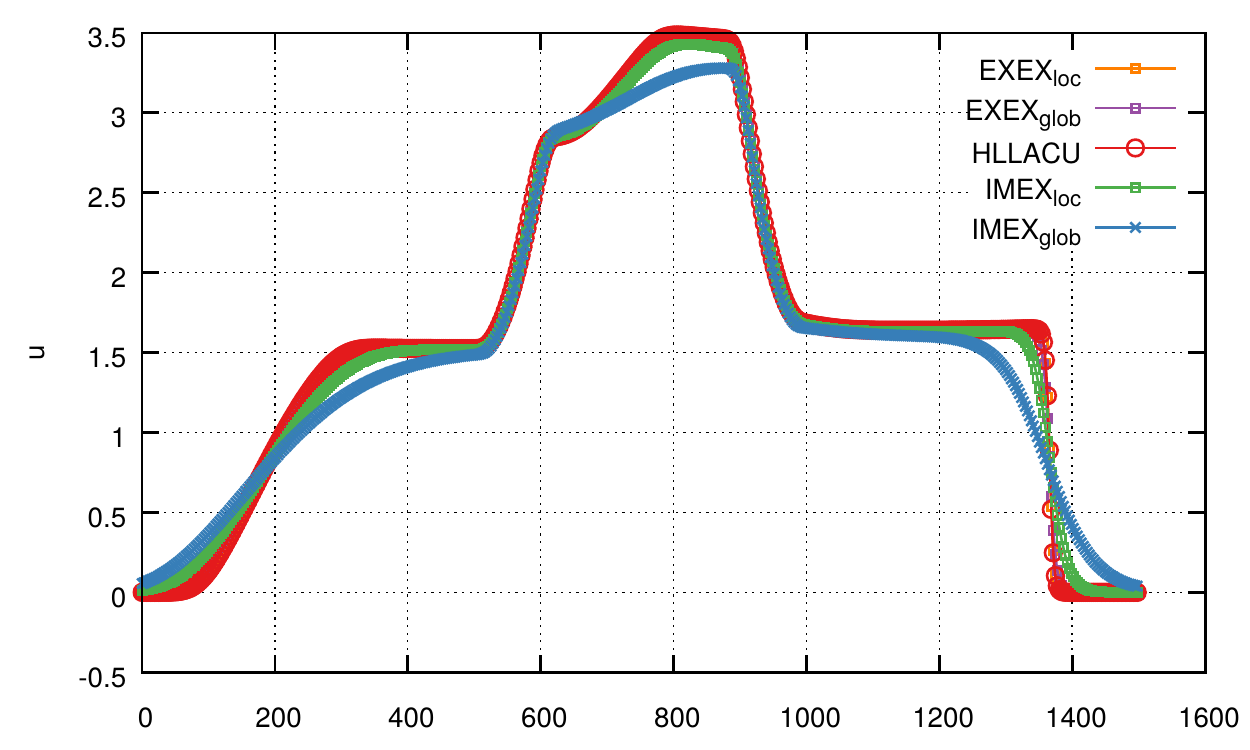}
\end{tabular}\caption{Dam Break problem. Profile of $z+h$ and $u$ at time $T=50$.}
\label{fig: DamBreak2}
\end{figure}

\subsection{Propagation of perturbations} 
This test case focuses on the perturbation of a steady state solution by a pulse that splits into two opposite waves. More precisely, the space domain is reduced to the interval $[0,2]$, the bottom topography is defined by
$z(x)=2+0.25(\cos(10\pi(x-0.5))+1)$ { if } $1.4<x<1.6$, and $2$ otherwise, and the 
initial state is such that
$u(0,x)=0$ and
$h(0,x)=3-z(x)+\Delta h$ if $1.1<x<1.2$, and $3-z(x)$ otherwise,
where $\Delta h = 0.001$ is the height of the perturbation.  
The CFL parameter is set to 0.9 (instead of $1/2$ in (\ref{cflnum1}) 
and (\ref{cflnum2})), the final time is $T=0.2$, the space step equals $\Delta x = {1}/{500}$ and Neumann boundary conditions are used.%
It turns out that since the perturbation is small, the values of the velocity $u$ keeps a small amplitude during the whole computation. As an immediate consequence,  
considering the natural CFL condition (\ref{cflnum2}) gives very large time steps 
which naturally induces much numerical diffusion. In order to reduce the numerical 
diffusion and improve the overall accuracy of the numerical solution, the time step 
$\Delta t_{imp}$
given by (\ref{cflnum2}) was first limited to ten times the time step $\Delta t_{exp}$
given by (\ref{cflnum1}). In other words, we chose
$$
\Delta t = \min(10 \Delta t_{exp}, \Delta t_{imp}) 
$$
for this test case.
Figure~\ref{fig: Perturb} compares the numerical solutions given by the \EXEXLOC{}, \IMEXLOC{} and HLLACU schemes. The implicit scheme is clearly more diffusive than the explicit ones.
Note that the so-called reference solution is given by the solution of the HLLACU scheme on a 10000-cell grid.

Figure~\ref{fig: Perturb2} shows that same solutions but the implicit scheme is now run using the explicit CFL restriction (\ref{cflnum1}). As expected, the numerical approximation is more accurate and the numerical diffusion is significantly reduced.
At last, %
Figure~\ref{fig: Perturb3} shows the numerical solutions 
using a $10000$-cell grid. The schemes converge to the same solution. %

\begin{figure}
\centering
\begin{tabular}{cc}
\includegraphics[width=.49\textwidth]{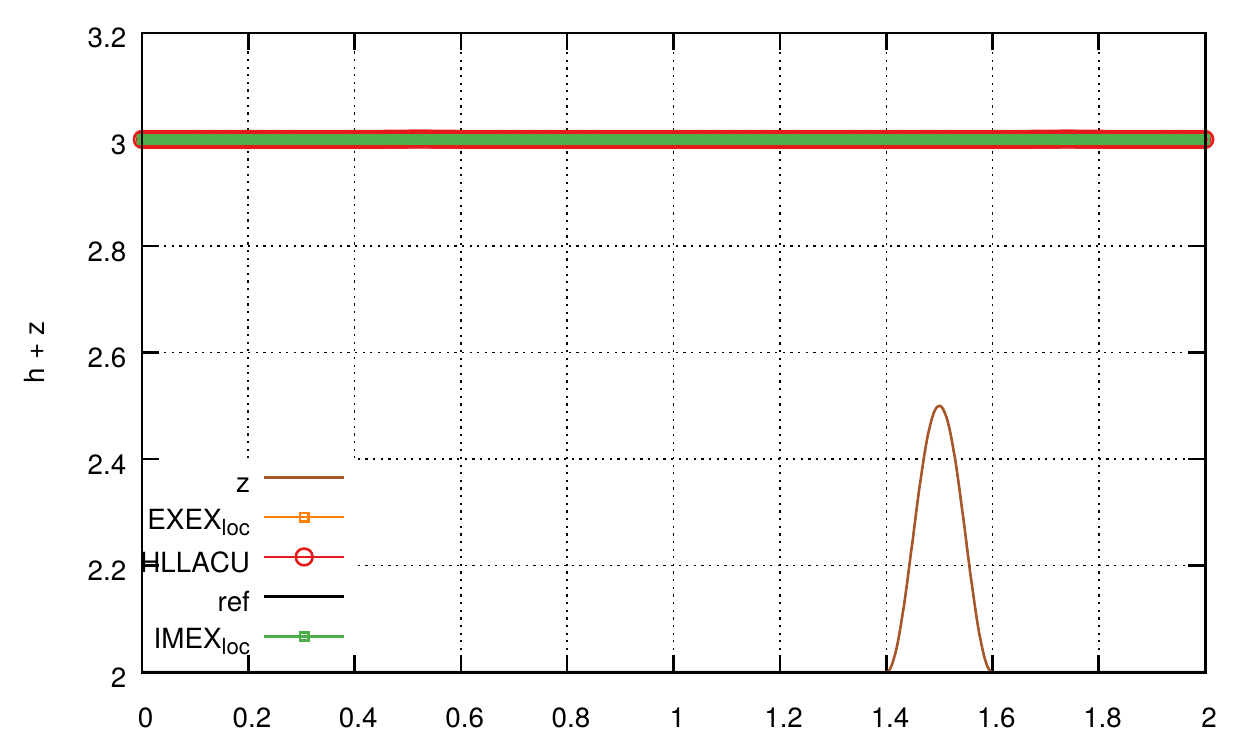}
&
\includegraphics[width=.49\textwidth]{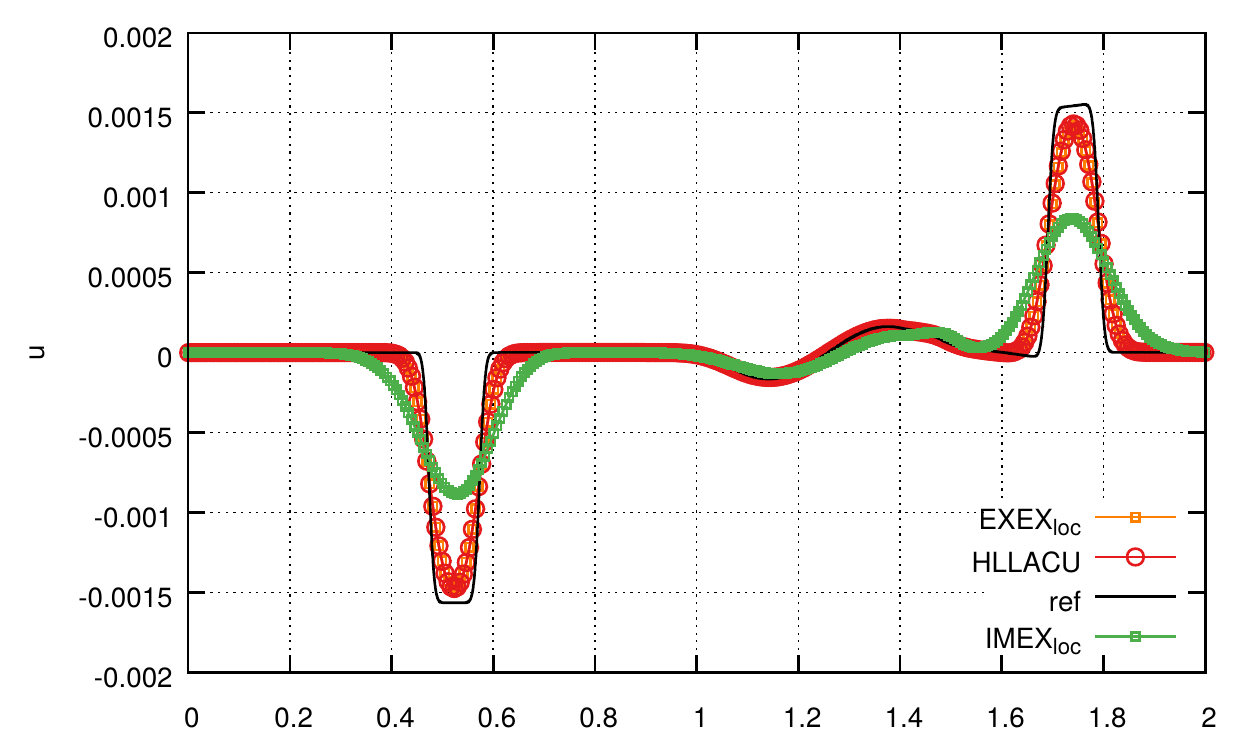}
\end{tabular}\caption{Propagation of perturbations test at final time $T=0.2$. On the left : total heights $h+z$, on the right : velocities $u$, with $\Delta x = 1/500$.}
\label{fig: Perturb}
\end{figure}

\begin{figure}
\centering
\begin{tabular}{cc}
\includegraphics[width=.49\textwidth]{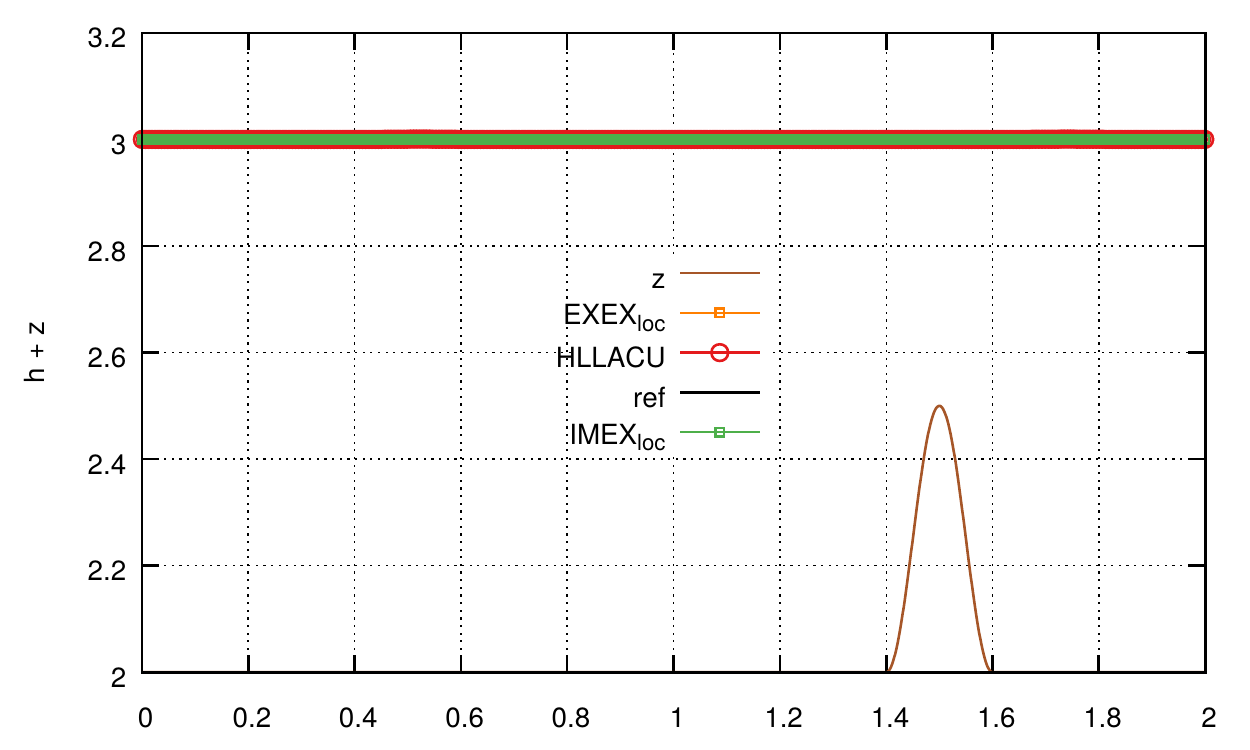}
&
\includegraphics[width=.49\textwidth]{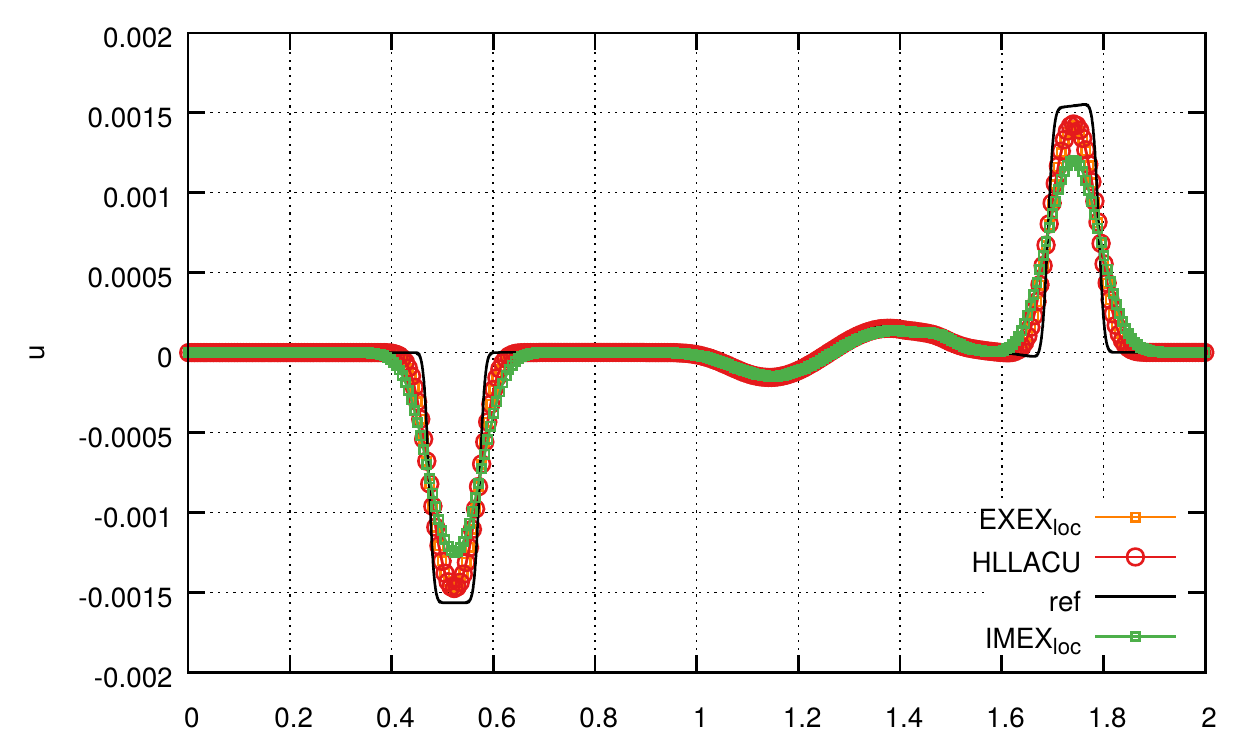}
\end{tabular}\caption{Propagation of perturbations test at final time $T=0.2$. On the left : total heights $h+z$, on the right : velocities $u$, with $\Delta x = 1/500$. Here, the implicit scheme is run using the explicit CFL restriction (\ref{cflnum1}).}
\label{fig: Perturb2}
\end{figure}

\begin{figure}
\centering
\begin{tabular}{cc}
\includegraphics[width=.49\textwidth]{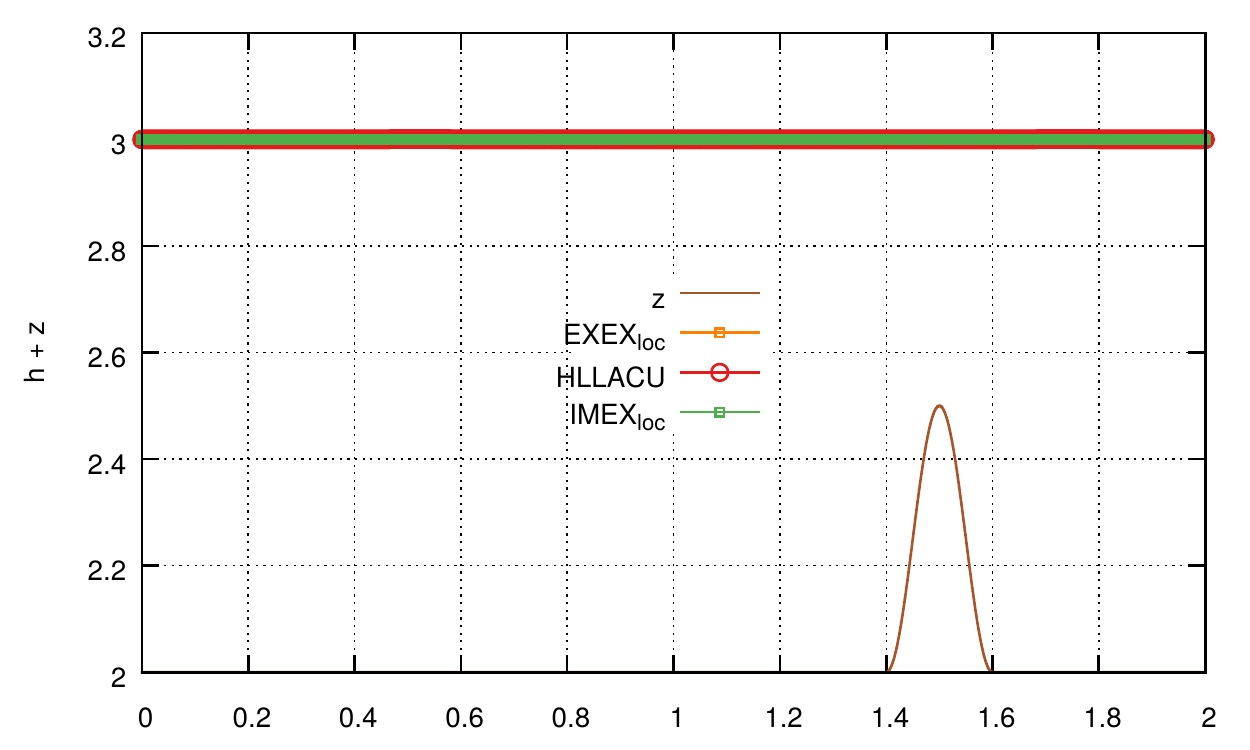}
&
\includegraphics[width=.49\textwidth]{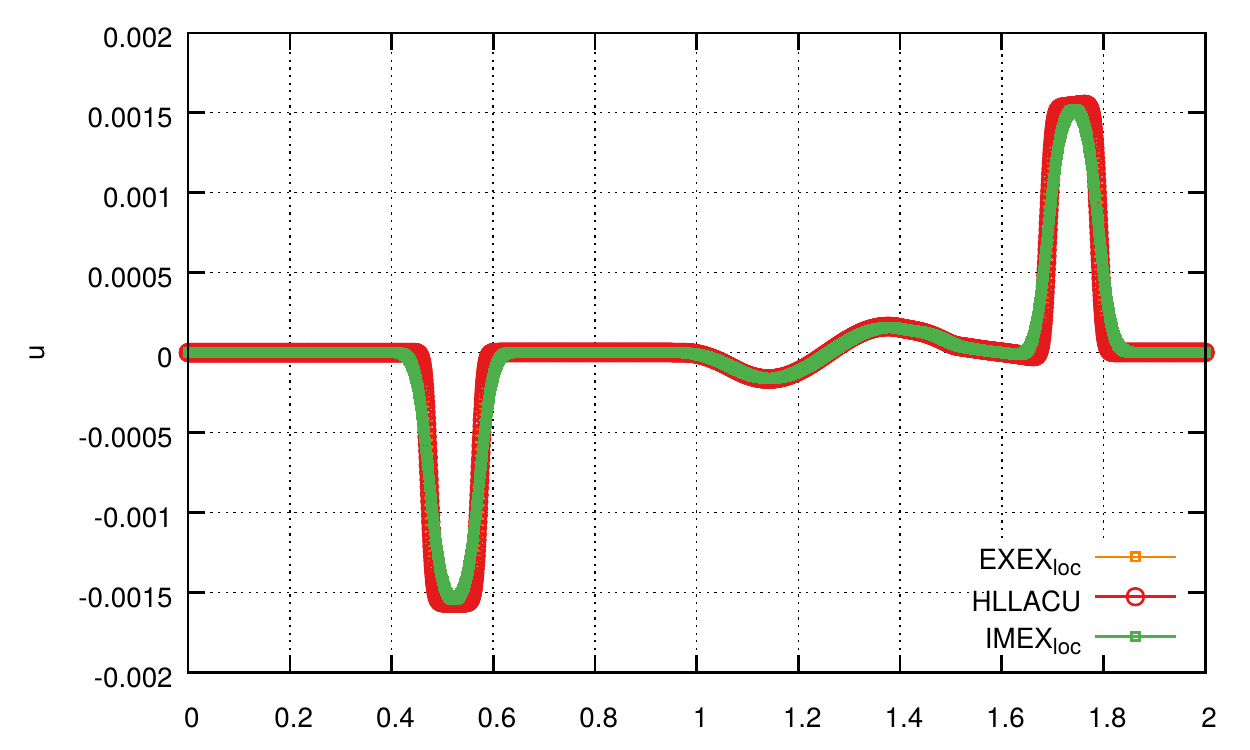}
\end{tabular}\caption{Propagation of perturbations test at final time $T=0.2$. On the left : total heights $h+z$, on the right : velocities $u$, with $\Delta x = 1/5000$.}
\label{fig: Perturb3}
\end{figure}

\subsection{Steady flow over a bump.}
The aim of this test case is to test the ability of the schemes to converge to some moving water equilibrium. %
Let us remind that the steady states are governed by the equations
$hu = K_1$ and $\displaystyle \frac{u^2}{2}+g(h+z)=K_2$.

\paragraph{Fluvial regime :} In this test case, we set $K_1=1$ and $K_2=25$, we denote $h_{eq}(x), u_{eq}(x)$ the values of $h$ and $u$ at this equilibrium. The domain is $[0,4]$ and the bottom topography is defined by
$z(x)=({\cos(10\pi(x-1))+1})/{4}$ if $1.9 \leq x \leq 2.1$ and $0$ elsewhere.
The CFL parameter is equal to 0.5 and the space step to $\Delta x = {1}/{400}$. %
The initial condition is chosen out of equilibrium and given by $h=h_{eq}$ and $u=0$. The boundary conditions are set to be
$$
\left\{
\begin{array}{l}
\partial_x h(x=0)= 0, \\
(hu)(x=0)=K_1,
\end{array}
\right.
\quad \mbox{and} \quad
\left\{
\begin{array}{l}
h(x=4)=h_{eq}(x=4), \\
\partial_x(hu)(x=4)= 0.
\end{array}
\right.
$$ 
Figure~\ref{fig: Fluvial} shows the solution at the final time $t=200$. We 
can observe that the solutions are close to the expected equilibrium, except near the mid domain where the momentum is not yet constant for the mesh size under consideration. The Lagrange-Projection schemes give numerical solutions very close to the one obtained with the HRHLL scheme based on the hydrostatic reconstruction, while the ACU scheme is clearly more accurate. Note also that on this test case, the implicit CFL condition 
(\ref{cflnum2}) allows to use time steps up to ten times larger than the explicit 
condition (\ref{cflnum1}). 
\begin{figure}
\centering
\begin{tabular}{cc}
\includegraphics[width=.49\textwidth]{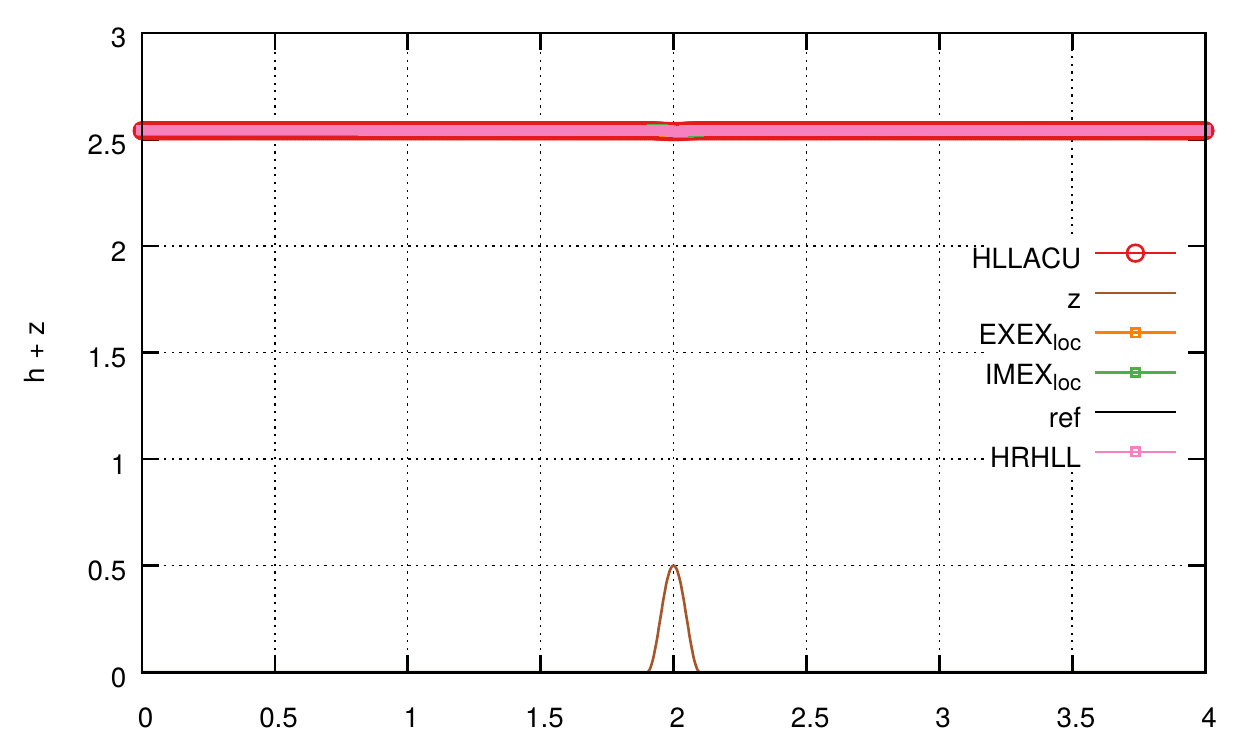}
&
\includegraphics[width=.49\textwidth]{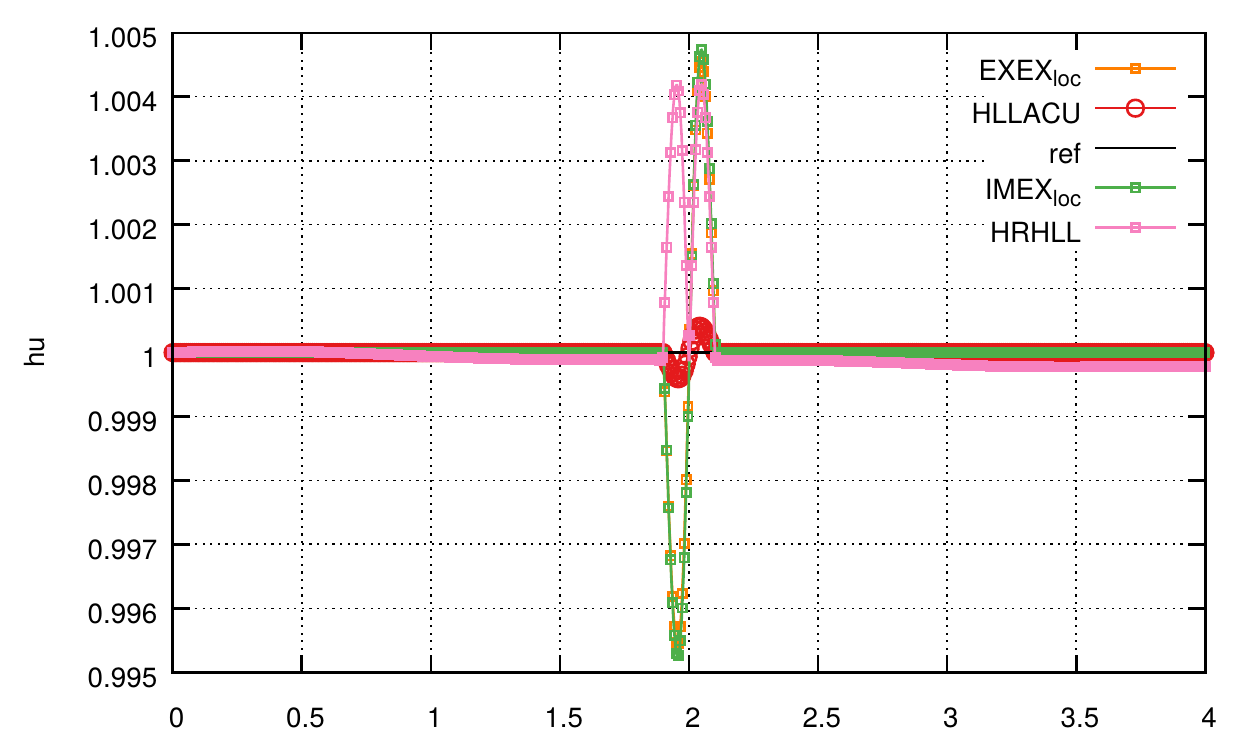}
\end{tabular}
\caption{Fluvial regime at time $T=200$. On the left : total heights $h+z$, on the right : discharge $hu$.}
\label{fig: Fluvial}
\end{figure}

\paragraph{Transcritical regime without shock :} In this test case, we set $K_1=3$, $K_2=\frac{3}{2}(K_1 g)^{2/3}+\frac{g}{2}$. We used the same boundary conditions and started from the same initial condition as in the previous simulation.
The solutions are shown at time $t=10$ on Figure~\ref{fig: TransSteady}. 

\begin{figure}
\centering
\begin{tabular}{cc}
\includegraphics[width=.49\textwidth]{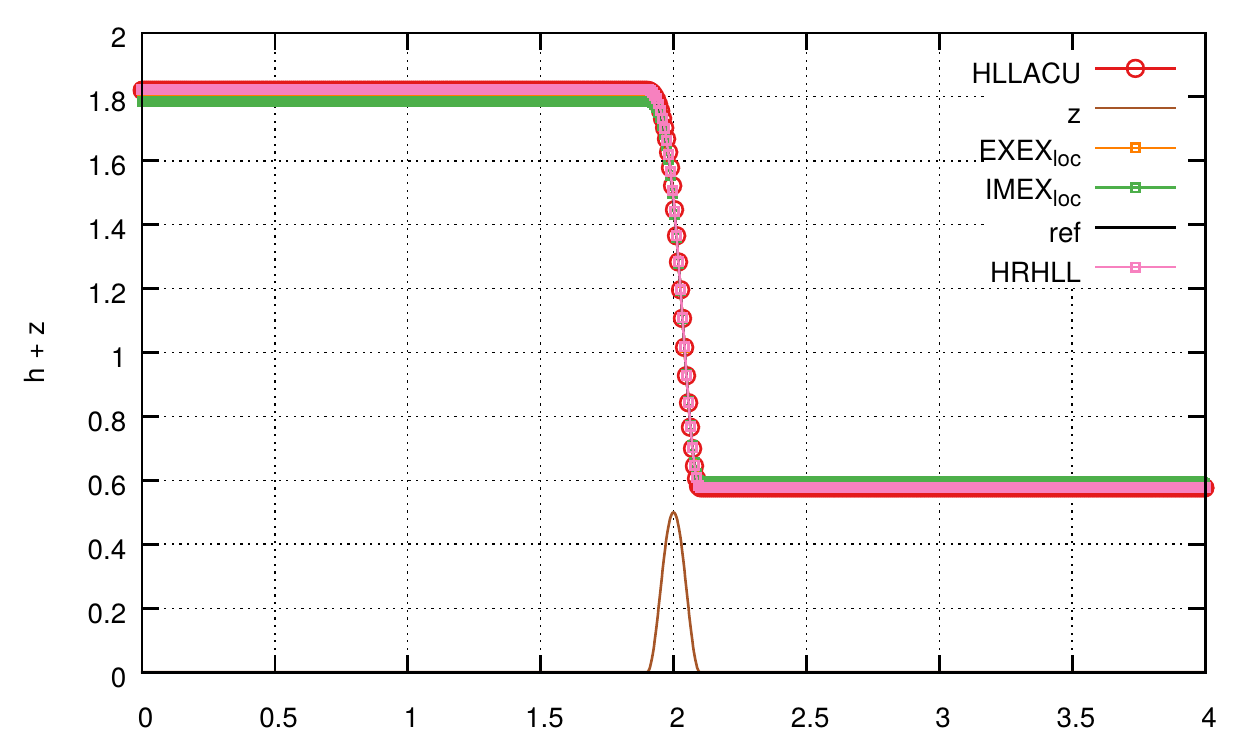}
&
\includegraphics[width=.49\textwidth]{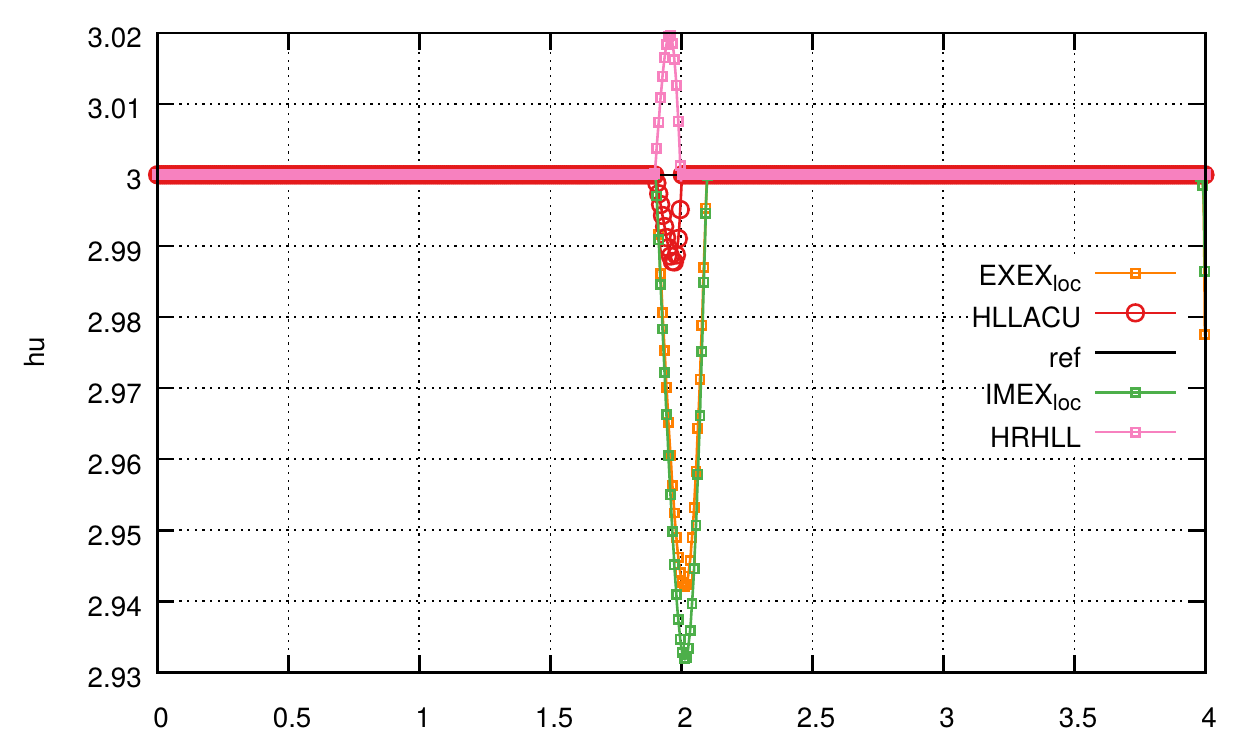}
\end{tabular}
\caption{Transcritical regime without shock. On the left : total heights $h+z$, on the right : discharge $hu$.}
\label{fig: TransSteady}
\end{figure}

\paragraph{Transcritical regime with shock :} This test has been proposed by Castro et al. \cite{12}. The parameters are described hereafter: the space domain is the interval $\left[0,25\right]$, the bottom topography is defined by $z(x)=3-0.005 (x-10)^2$, if $8<x<12$, and $2.8$ otherwise. The initial state is defined by $h(0,x)=3.13-z(x)$, $q(0,x)=0.18$ and the boundary conditions are $q(t,0)=0.18$, $\partial_x q(t,25)=0$, $h(t,25)=0.33$ and $\partial_x h(t,0)=0$. The final time is set to $t=200$, the space step to $\Delta x = {1}/{64}$ and the CFL to $0.9$.
We can see on the Figure~\ref{fig: Transcritical} 
that we obtain similar results with the different schemes.
\begin{figure}
\centering
\begin{tabular}{cc}
\includegraphics[width=.49\textwidth]{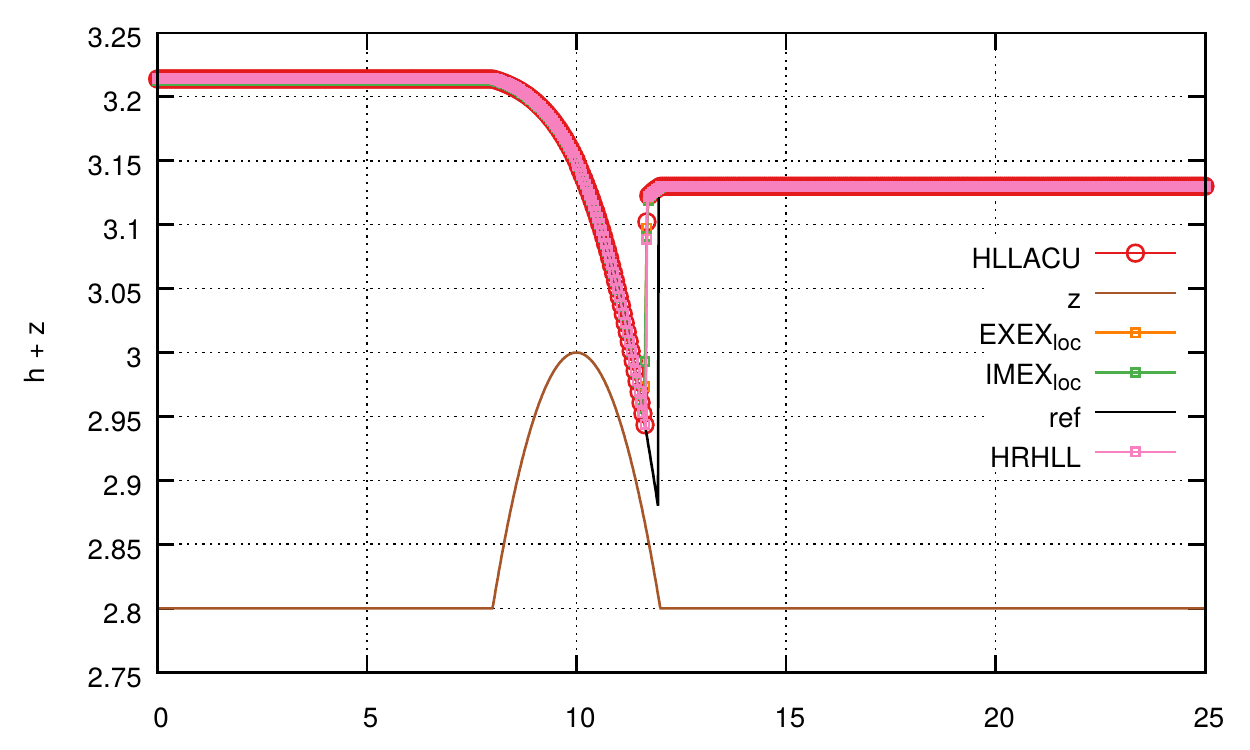}
&
\includegraphics[width=.49\textwidth]{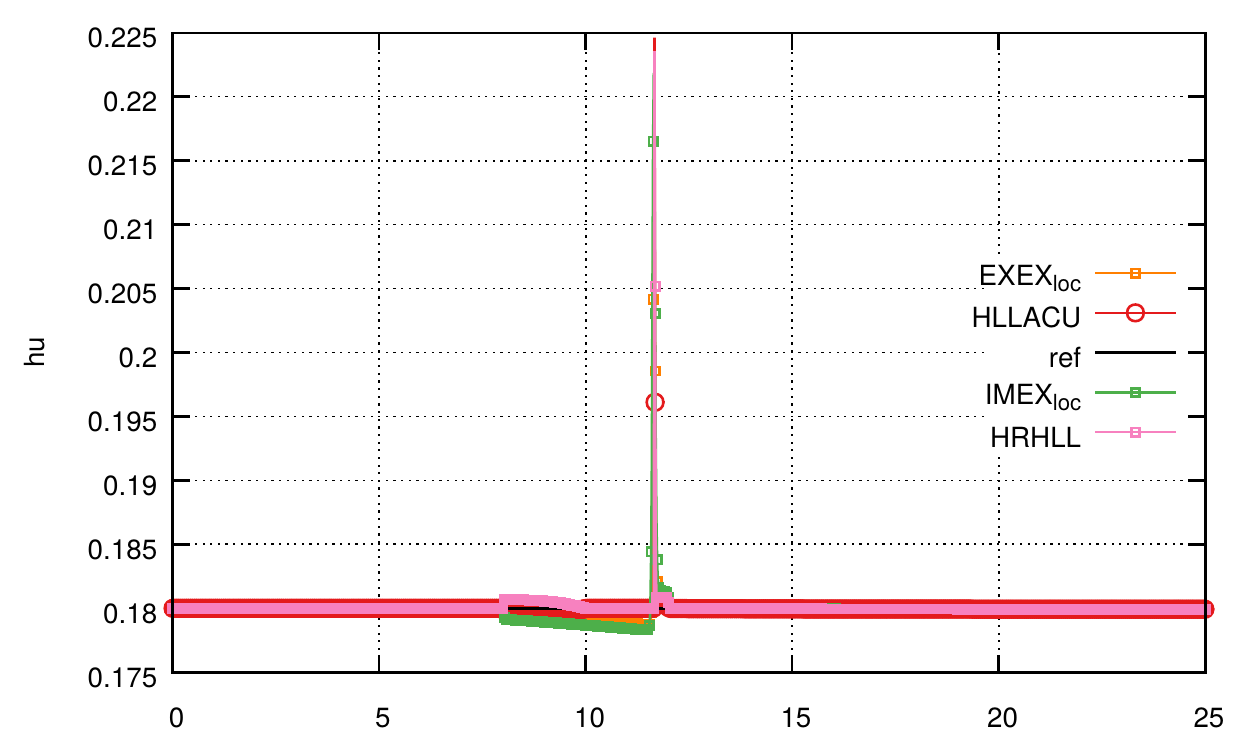}
\end{tabular}
\caption{Transcritical regime with shock at final time $T=200$. On the left : total heights $h+z$, on the right : discharge $hu$.}
\label{fig: Transcritical}
\end{figure}

\subsection{Non-unique solution to the Riemann problem}

This aim of this test case is to consider a Riemann problem for which the entropy solution is not unique, in order to see whether the numerical schemes capture the same solution or not. The spatial domain is $\left[0,1\right]$, the gravitational acceleration $g$ is set to $2$ and the CFL coefficient equals $0.9$. Note however that considering the mixed implicit-explicit scheme, the time step $\Delta t$ was restricted to three times the explicit time step, namely
$$
\Delta t = \min(3 \Delta t_{exp}, \Delta t_{imp}) 
$$
where we have used the same notations as in the propagation of perturbations test case.
The final time $T=0.1$ and the space step is $\Delta x = {1}/{300}$.
The initial data is given by
\begin{equation*}
(z,h,u)^T = 
\begin{cases}
(1.5,1.3,-2)^T & \text{ if } x\leq 0.5,\\
(1.1,0.1,-2)^T & \text{ if } x>0.5,
\end{cases}
\end{equation*} 
and we used Neumann boundary conditions.
It is quite interesting to observe on Figure~\ref{fig: NonUnique} that the methods proposed in the present paper and the hydrostatic scheme seem to converge to the same solution, while the HLLACU scheme capture a quite different solution. 

\begin{figure}
\centering
\begin{tabular}{cc}
\includegraphics[width=.49\textwidth]{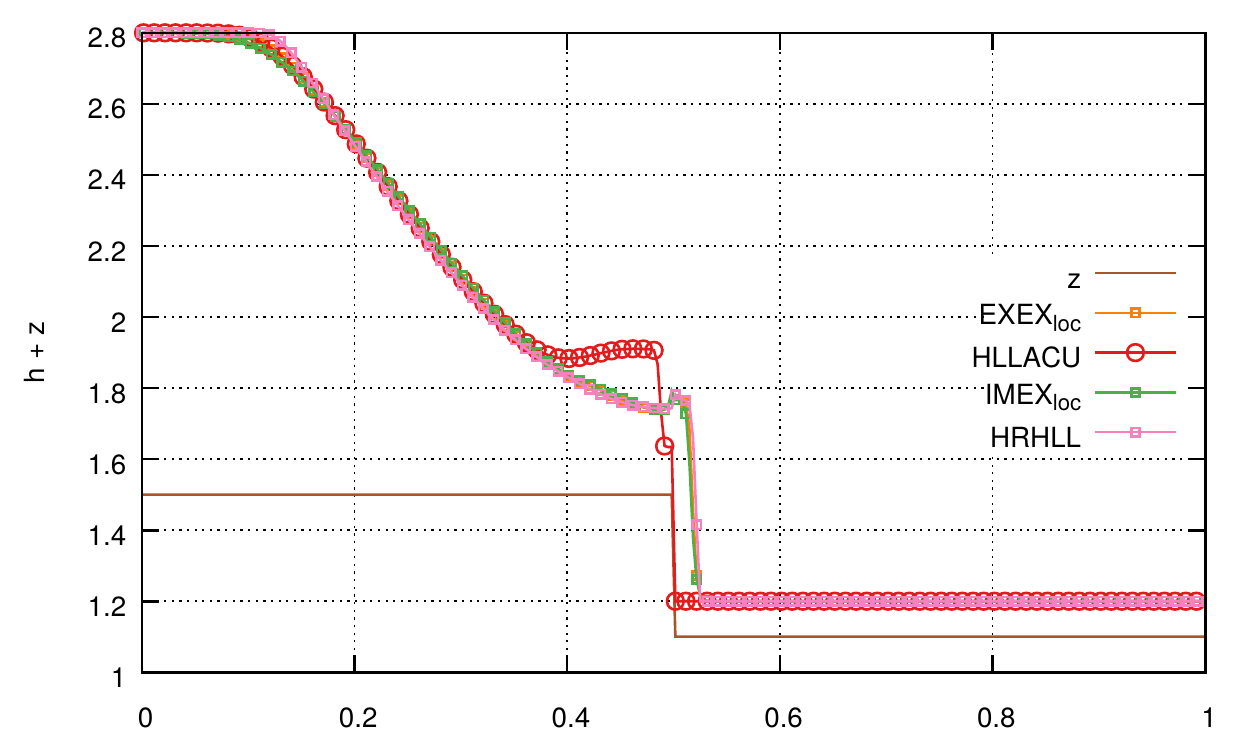}
&
\includegraphics[width=.49\textwidth]{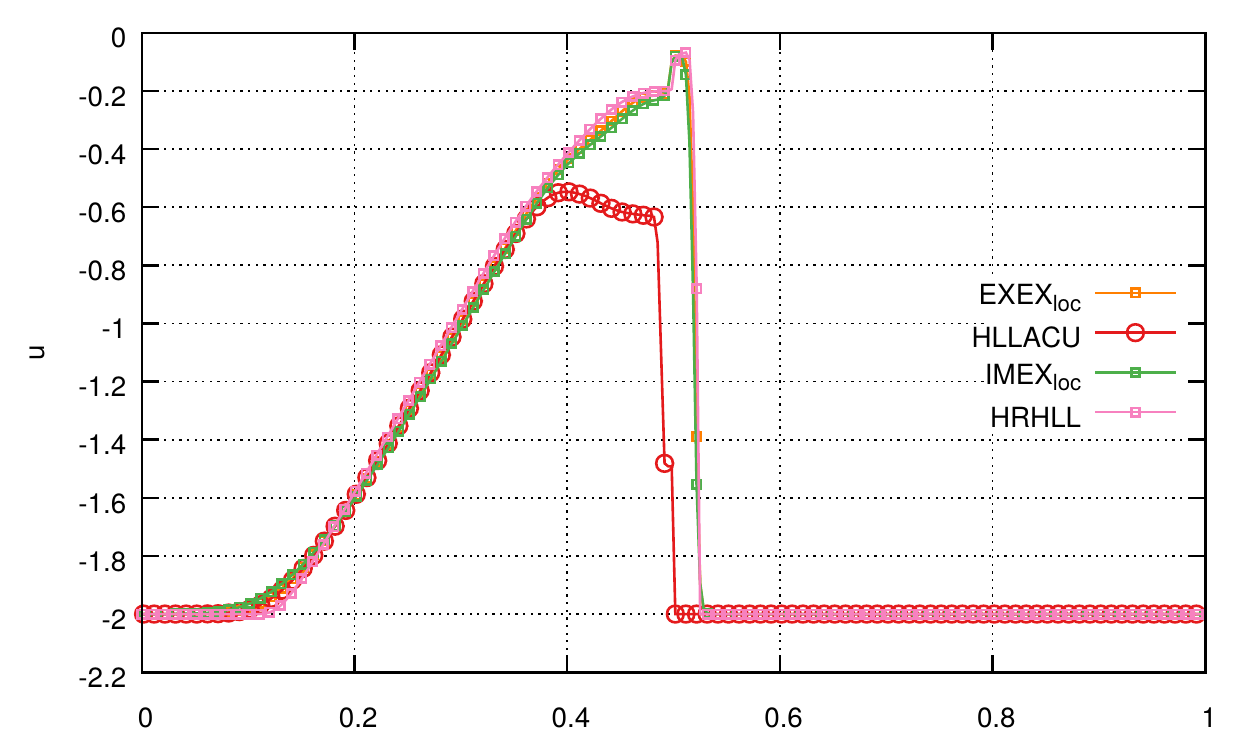}
\end{tabular}
\caption{Non-unique solution test case at final time $T=0.1$. On the left : total heights $h+z$, on the right : velocities $u$.}
\label{fig: NonUnique}
\end{figure}

\makeatletter{}%
\section*{Conclusion}

We have proposed a large time step and well-balanced scheme for the 
shallow-water equations and proved stability properties 
under a time step CFL restriction based on the material velocity $u$ and not on the sound speed $c$ as it is customary. The Lagrangian-Projection decomposition 
proved to be efficient on a variety of test cases, but may be more diffusive than a 
direct Eulerian approach. We believe that the proposed implicit-explicit strategy is especially well adapted for subsonic flows but even more for large Froude numbers, which is our very motivation and the purpose 
of an ongoing work in several space dimensions. Works in progress also include a high-order accuracy extension using discontinuous Galerkin strategies for the space variable and Runge-Kunta techniques for the time variable. \\
\ \\
{\bf Acknowledgement}. This work was partially supported by a public grant as part of the
Investissement d’avenir project, reference ANR-11-LABX-0056-LMH, LabEx LMH. The authors also thank S. Noelle and H. Zakerzadeh for useful and interesting discussions on the topic. %

\makeatletter{}%
\makeatletter{}%
\appendix

\section{Eigenstructure of the relaxed acoustic system}
\label{appendix: eigenstructure acoustic relaxed}
Considering smooth solutions, the homogeneous relaxed acoustic system~\eqref{eq: acoustic m relaxation}  reads

\begin{equation}
\partial_t \W + A(W) \partial_m\W =0
,\quad
A = 
\begin{pmatrix}
0 & -1 & 0 & 0
\\
0 & 0 & 1 & g / \tau
\\
0 & a^2 & 0 & 0
\\
0 & 0 & 0 & 0
\\
\end{pmatrix}
.
\label{eq: acoustic relaxed m quasilinear}
\end{equation}
The matrix the eigenvalues of $A$ are $\{-a, 0,+a \}$. A basis of right eigenvectors of $A$ is
$$
\br_0^{(1)} = (1,0,0,0)^T
,\qquad
\br{}_0^{(2)} = (0,0,-g,\tau)^T
,\qquad
\br_{\pm} = (1,\mp a, -a^2,0)^T
$$
where $\br_0^{(1)}$ and $\br_0^{(1)}$ are associated with the double eigenvalue $0$ and  $\br_{\pm}$ is associated with $\pm a$. The system~\eqref{eq: acoustic relaxed m quasilinear} is thus hyperbolic. All the characteristic fields of \eqref{eq: acoustic relaxed m quasilinear} are linearly degenerate.

The $(\pm a)$-field possesses three Riemann invariants
$$
I^{1}_{\pm} = \pi \mp au
,\qquad
I^{2}_{\pm} = u \pm a \tau
,\qquad
I^{3}_{\pm} = z
.
$$
As a consequence, the states $\W_L$ and $\W_R$ that can be connected by a $(\pm a)$-wave can be obtained thanks to the continuity of the $(\pm a)$-Riemann invariants, which amounts to verify the jump relations
\begin{equation}
\mp a (\W_R - \W_R) + \G(\W_R) - \G(\W_R) = 0.
\label{eq: jump relation wave pm a}
\end{equation}
Unfortunately, the eigenvalue $0$ is of multiplicity 2 and the $0$-field only has a single Riemann invariant
$$
I_0 = u
.
$$
Therefore we can only state that if two states $\W_L$ and $\W_R$ are connected by a $0$-wave then
\begin{equation}
u_R = u_L.
\label{eq: jump relation wave 0}
\end{equation}

\section{Proof of the discrete entropy inequality of Proposition \ref{GIRpart3thrm2}}
\label{appendix:proof_proposition_implicite}

The proof of the discrete entropy inequality follows exactly the same lines as the one proposed in \cite{quote6} for the barotropic gas dynamics equations, but taking into account here the presence of the topography source term. It naturally leads to a non conservative version of the entropy inequality. A discrete and conservative entropy inequality for the proposed algorithm remains an open problem so far.
Our result states as follows.

\begin{lemma}
We have the following discrete form of the entropy inequality \eqref{entropyIneq} for all $j\in \mathbb{Z}$, namely
\[
\mathcal{U}_j^{n+1} - \mathcal{U}_j^{n} + \frac{\Delta t}{\Delta x_j} 
 \left(\mathcal{F}_{j+1/2}^{n+1-}-\mathcal{F}_{j-1/2}^{n+1-}\right) \leq  - \Delta t \, g \left\{hu\partial_xz\right\}_j ,
\] 
with the entropy numerical fluxes
\[
\mathcal{F}_{j+1/2}^{n+1-} = \left({\Pi}_{j+1/2}^* +  \mathcal{U}^{n+1-}_{j+1/2}\right) \tilde{u}^{*}_{j+1/2},
\]
where 
\[
\mathcal{U}^{n+1-}_{j+1/2}=\begin{cases}
\mathcal{U}^{n+1-}_{j} & \text{if } u^{*}_{j+1/2}\geq 0,\\
\mathcal{U}^{n+1-}_{j} & \text{if } u^{*}_{j+1/2} < 0,
\end{cases}
\] 
and 
\[
\left\{
\begin{aligned}
\tilde{u}^*_{j+1/2} &= \frac{u_j^{n+1-}+u_{j+1}^{n+1-}}{2} - \frac{1}{2a} \left(\Pi_{j+1}^{n+1-}-\Pi_j^{n+1-}\right), \\
{\Pi}^*_{j+1/2} &= \frac{\Pi_j^{n+1-}+\Pi_{j+1}^{n+1-}}{2} - \frac{a}{2} \left(u_{j+1}^{n+1-}-u_j^{n+1-}\right),
\end{aligned}
\right.
\]
are consistent with $\mathcal{F}$, 
and the non conservative source term 
\[
\left\{hu\partial_xz\right\}_j = \frac{1}{2 a \Delta x_j} \left[\frac{h_{j}^n+h_{j-1}^n}{2} \overrightarrow{w}_j^{n+1-} \left(z_{j}-z_{j-1}\right) - \frac{h_{j+1}^n+h_j^n}{2} \overleftarrow{w}_j^{n+1-} \left(z_{j+1}-z_j\right) \right]
\]
is consistent with $hu \partial_x z$.
\end{lemma}

\begin{proof}\rule{0pt}{0pt}
Let us first observe that smooth solutions of 
\eqref{eq: acoustic m relaxation} satisfy %
\begin{equation} \label{eq: equality proof discrete ineq}
 \partial_t (\Pi^2 + a^2 u^2) + 2 a^2 \partial_m \Pi u = - \frac{2a^2gu}{\tau} \partial_m z.
\end{equation}
In order to obtain a discrete version of this equality, let us define 
$$
\eta_j^{n+1-} := \frac{(\overleftarrow{w}_j^{n+1-})^2 + (\overrightarrow{w}_j^{n+1-})^2}{2} = (\Pi_j^{n+1-})^2 + a^2 (u_j^{n+1-})^2,
$$
and
$$
q_{j+1/2}^{n+1-} := \frac{(\overrightarrow{w}_j^{n+1-})^2 - (\overleftarrow{w}_{j+1}^{n+1-})^2}{4a} = \Pi_{j+1/2}^* \, \tilde{u}_{j+1/2}^*.
$$
The formulas (\ref{eq: implicit acoustic characteristic}) also read
$$
\left\{
\begin{array}{l}
\displaystyle   \tau_j^{n+1^-} - \tau_j^n = \frac{\Delta t}{\Delta m_j} \left[ u_{j+1/2}^* - u_{j-1/2}^* \right],
  \\
  \displaystyle   \overleftarrow{w}_j^{n+1-} - \overleftarrow{w}_j^n = a \frac{\Delta t}{\Delta m_j} \left[\overleftarrow{w}^{n+1-}_{j+1} - \overleftarrow{w}^{n+1-}_{j} + g \frac{h_{j+1}^n+h_j^n}{2} \left(z_{j+1}-z_j\right) \right],
  \\
\displaystyle   \overrightarrow{w}_j^{n+1-} - \overrightarrow{w}_j^n =- a \frac{\Delta t}{\Delta m_j} \left[\overrightarrow{w}^{n+1-}_{j} - \overrightarrow{w}^{n+1-}_{j-1}  + g \frac{h_{j}^n+h_{j-1}^n}{2} \left(z_{j}-z_{j-1}\right) \right],
 \end{array}
\right.
$$
while adding the third equation of (\ref{eq: update acoustic implicit scheme W}) and $a^2$ times the first equation of (\ref{eq: update acoustic implicit scheme W}) also gives $I_j^{n+1-} = I_j^n$ where $I=\Pi+a^2 \tau$. Multiplying the second and the third equations above by 
$\overleftarrow{w}_j^{n+1-}$ and $ \overrightarrow{w}_j^{n+1-}$ then gives
$$
\left\{
\begin{array}{l}
\displaystyle   I_j^{n+1-} = I_j^n,
  \\
\displaystyle   \overleftarrow{w}_j^{n+1-} (\overleftarrow{w}_j^{n+1-} - \overleftarrow{w}_j^n) = a \frac{\Delta t}{\Delta m_j} \left[\overleftarrow{w}_j^{n+1-} \left(\overleftarrow{w}^{n+1-}_{j+1} - \overleftarrow{w}^{n+1-}_{j}\right) + \overleftarrow{w}_j^{n+1-} g \frac{h_{j+1}^n+h_j^n}{2} \left(z_{j+1}-z_j\right) \right],
  \\
\displaystyle   \overrightarrow{w}_j^{n+1-}(\overrightarrow{w}_j^{n+1-} - \overrightarrow{w}_j^n) =- a \frac{\Delta t}{\Delta m_j} \left[\overrightarrow{w}_j^{n+1-}\left(\overrightarrow{w}^{n+1-}_{j} - \overrightarrow{w}^{n+1-}_{j-1}\right)  + \overrightarrow{w}_j^{n+1-}  g \frac{h_{j}^n+h_{j-1}^n}{2} \left(z_{j}-z_{j-1}\right) \right],
 \end{array}
\right.
$$
that is to say, since 
$$
2b(b-a)=(b^2-a^2) + (b-a)^2 \quad \mbox{and} \quad
2b(a-b)=(a^2-b^2) - (b-a)^2,
$$
$$
\left\{
\begin{array}{l}
\displaystyle   I_j^{n+1-} = I_j^n,
  \\
\displaystyle 
\left((\overleftarrow{w}_j^{n+1-})^2 - (\overleftarrow{w}_j^{n})^2\right) - a \frac{\Delta t}{\Delta m_j}\left( (\overleftarrow{w}^{n+1-}_{j+1})^2 - (\overleftarrow{w}^{n+1-}_{j})^2\right)= \\
\qquad - (\overleftarrow{w}_j^{n+1-} - \overleftarrow{w}_j^{n})^2
+ a \frac{\Delta t}{\Delta m_j}\left[-\left(\overleftarrow{w}^{n+1-}_{j+1} - \overleftarrow{w}^{n+1-}_{j}\right)^2 + 2 \overleftarrow{w}_j^{n+1-} g \frac{h_{j+1}^n+h_j^n}{2} \left(z_{j+1}-z_j\right) \right],
  \\
\displaystyle   
\left((\overrightarrow{w}_j^{n+1-})^2 - (\overrightarrow{w}_j^{n})^2\right) + a \frac{\Delta t}{\Delta m_j}\left((\overrightarrow{w}^{n+1-}_{j})^2 - (\overrightarrow{w}^{n+1-}_{j-1})^2\right)= \\
\qquad-(\overrightarrow{w}_j^{n+1-} - \overrightarrow{w}_j^{n})^2 -
a \frac{\Delta t}{\Delta m_j}\left[\left(\overrightarrow{w}^{n+1-}_{j} - \overrightarrow{w}^{n+1-}_{j-1}\right)^2 + 2 \overrightarrow{w}_j^{n+1-} g \frac{h_{j}^n+h_{j-1}^n}{2} \left(z_{j}-z_{j-1}\right)\right].
 \end{array}
\right.
$$
Summing the last two equations, we immediately get the following 
discrete version of \eqref{eq: equality proof discrete ineq}, namely
$$
\eta_j^{n+1-} - \eta_j^{n} + 2 a^2 \frac{\Delta t}{\Delta m_j} 
(q_{j+1/2}^{n+1-}-q_{j-1/2}^{n+1-}) \leq - \Delta t \, 2a^2 \, g \, \tau_j^n \left\{hu\partial_xz\right\}_j.
$$ 
The rest of the proof strictly follows the one proposed in \cite{quote6}. It is given here for the sake of completeness. With this in mind, let us define the energy $E$ such that $h E = \mathcal{U}$, which means 
$$
E = \frac{u^2}{2} + e(\tau) = \frac{u^2}{2} + e(\tau) + \frac{\Pi^2-\Pi^2}{2a^2} = e(\tau) +
\frac{\eta-\Pi^2}{2a^2}, 
$$
where we have set $e(\tau)=\frac{g}{2\tau}=\frac{gh}{2}$.
We clearly have
$$
E_{j}^{n+1-} - E_{j}^{n} = 
e(\tau_{j}^{n+1-}) - e(\tau_{j}^{n}) +
\frac{\eta_{j}^{n+1-}-\eta_{j}^{n}}{2a^2} - 
\frac{(\Pi_{j}^{n+1-})^2-(\Pi_{j}^{n})^2}{2a^2}
$$
so that, since $a^2-b^2=(a-b)^2+2b(b-a)$, we have
$$
E_{j}^{n+1-} - E_{j}^{n} = 
e(\tau_{j}^{n+1-}) - e(\tau_{j}^{n}) +
\frac{\eta_{j}^{n+1-}-\eta_{j}^{n}}{2a^2} - 
\frac{(\Pi_{j}^{n+1-}-\Pi_{j}^{n})^2}{2a^2}- 
\frac{\Pi_{j}^{n}(\Pi_{j}^{n+1-}-\Pi_{j}^{n})}{a^2}.
$$
But $I_j^{n+1-} = I_j^n$ gives $\Pi_{j}^{n+1-}-\Pi_{j}^{n} = -a^2 (\tau_{j}^{n+1-}-\tau_{j}^{n})$ so that 
$$
E_j^{n+1-} - E_j^{n} + \frac{\Delta t}{\Delta m_j} 
(q_{j+1/2}^{n+1-}-q_{j-1/2}^{n+1-}) \leq 
e(\tau_{j}^{n+1-}) - e(\tau_{j}^{n}) +
{\Pi_{j}^{n}(\tau_{j}^{n+1-}-\tau_{j}^{n})}
-\frac{a^2}{2}(\tau_{j}^{n+1-}-\tau_{j}^{n})^2 - \Delta t \, g \, \tau_j^n \left\{hu\partial_xz\right\}_j.
$$ 
Since the solution at time $t^n$ is at equilibrium, we have 
$\displaystyle \Pi^n_j=p(\tau^n_j)=-e'(\tau^n_j)=\frac{g}{2}(h^n_j)^2$, so that a Taylor expansion gives
$$
E_j^{n+1-} - E_j^{n} + \frac{\Delta t}{\Delta m_j} 
(q_{j+1/2}^{n+1-}-q_{j-1/2}^{n+1-}) \leq 
\frac{(e^{''}(\xi) - a^2)}{2}(\tau_{j}^{n+1-}-\tau_{j}^{n})^2 - \Delta t \, g \, \tau_j^n \left\{hu\partial_xz\right\}_j
$$ 
and
$$
E_j^{n+1-} - E_j^{n} + \frac{\Delta t}{\Delta m_j} 
(q_{j+1/2}^{n+1-}-q_{j-1/2}^{n+1-}) \leq 
\frac{(-p^{'}(\xi) - a^2)}{2}(\tau_{j}^{n+1-}-\tau_{j}^{n})^2 - \Delta t \, g \, \tau_j^n \left\{hu\partial_xz\right\}_j \leq - \Delta t \, g \, \tau_j^n \left\{hu\partial_xz\right\}_j
$$ 
by the Whitham subcharacteristic condition. This inequality is nothing but the expected entropy inequality but in Lagrangian coordinates. At this stage, it is very usual to combine 
the definition of the remap step (which, setting $X=h, h u$, gives $X^{n+1}_j$ 
as a convex combination of $X^{n+1-}_{j-1}$, $X^{n+1-}_{j}$ and $X^{n+1-}_{j+1}$ under the transport CFL condition) together with the Jensen inequality for 
the convex mapping $(h,hu)\mapsto \mathcal{U}(h,hu)$, in order to get the expected entropy inequality in Eulerian coordinates, namely
\[
\mathcal{U}_j^{n+1} - \mathcal{U}_j^{n} + \frac{\Delta t}{\Delta x_j} 
 \left(\mathcal{F}_{j+1/2}^{n+1-}-\mathcal{F}_{j-1/2}^{n+1-}\right) \leq  - \Delta t \, g \left\{hu\partial_xz\right\}_j.
\] 
We refer the reader to \cite{quote6} for more details. 
\\

\end{proof} %

\end{document}